\documentclass[final,leqno,onefignum,onetabnum]{siamltex1213}

\usepackage{amssymb,amsmath,cases}
\newcommand{\ol}{\overline}
\newcommand{\nn}{\nonumber}
\newcommand{\idc}[1]{{\bf 1}_{\{#1\}}}

\newtheorem{assumption}{Assumption}
\newtheorem{remark}{Remark}

\title{Quickest detection in coupled systems\thanks{An earlier version of this work was presented at the 48th IEEE Conference on Decision and Control and appears in citation \cite{HadjSchaPoor} of this paper. This
    work was partially supported by the RF-CUNY Collaborative grant 80209-04 15, the PSC-CUNY grant 65625-00 43,
    the NSA-MSP grant 081103, the NSF-DMS grant 0807396, the
    NSF-CNS grant 0855217, the ONR grant N00014-09-1-0342, the NSF-CCF-MSC grant 0916452, the NSF-DMS grant 0929317, the NSF-DMS grant 1118605, and the NSF-DMS-ATD grant 1222526. TS gratefully acknowledges the support of
    the CUNY High Performance Computing Facility and the Center for
    Interdisciplinary Applied Mathematics and Computational Sciences.}} 

\author{
Hongzhong Zhang \thanks{Statistics Department, Columbia University (\email{hzhang@stat.columbia.edu})}
\and  Olympia Hadjiliadis \thanks{ Department of Mathematics, Brooklyn
    College, City University of New York, and Departments of Computer
    Science and Mathematics, Graduate Center, City University of New
    York (\email{ohadjiliadis@brooklyn.cuny.edu})}
    \and Tobias Sch\"afer
  \thanks{Department of Mathematics, College of Staten Island, City
    University of New York, and Department of Physics, Graduate
    Center, City University of New York (\email{tobias@math.csi.cuny.edu})}
 \and H.~Vincent Poor \thanks{Department of Electrical Engineering, Princeton
    University (\email{poor@princeton.edu})}
}

\begin{document}
\maketitle
\slugger{sicon}{xxxx}{xx}{x}{x--x}

\begin{abstract}
This work considers the problem of quickest
detection of signals in a coupled system of $N$ sensors, which receive
continuous sequential observations from the environment.
It is assumed that the signals, which are modeled by general
It\^{o} processes, are coupled across sensors, but that
their onset times may differ from sensor to sensor. Two main cases are considered; in the first one
signal strengths are the same across sensors while in the second one they differ by a constant.
The objective
is the optimal detection of the first time at
which any sensor in the system receives a signal.
The problem
is formulated as a stochastic optimization problem in which
an extended minimal Kullback-Leibler divergence criterion
is used as a measure of detection delay, with a constraint on
the mean time to the first false alarm.
The case in
which the sensors employ cumulative sum (CUSUM) strategies
is considered, and it is proved that the minimum of $N$ CUSUMs
is asymptotically optimal as the mean time to the first false alarm
increases without bound. In particular, in the case of equal signal
strengths across sensors, it is seen that the difference in detection
delay of the $N$-CUSUM stopping rule and the unknown optimal stopping
scheme tends to a constant related to the number of sensors as the
mean time to the first false alarm increases without bound.
Alternatively, in the case of unequal signal strengths, it is seen that this difference tends to zero.
\end{abstract}



\pagestyle{myheadings}
\thispagestyle{plain}
\markboth{Quickest detection}{Quickest detection}

\section{Introduction}

We are interested in the
problem of quickest detection of the onset of a signal in a system of $N$ sensors. We consider the situation in which, although
the observations in one sensor can affect the observations in another, the
onset of a signal can occur at different times (i.e., change points) in each of the $N$ sensors;
that is, the change points differ from sensor to sensor. As an example in which this situation arises
consider a system of sensors monitoring the health of a physical structure in which fault conditions are manifested by vibrations in the structure. Before a change
affects a given sensor, we have only noise in that sensor. Then, after a change, the system is vibrating and
thus the signal received in any location reflects a vibrating system.
Thus, observations at any given sensor are coupled with those received in other locations.
The change points observed at different sensors can occur at different times because the source of the vibrations
(i.e., the excitation) may arrive at different structural elements at different times. Relevant literature related to such models includes, for example, \cite{BassAbdeBenv,BassBenvGourMeve,BassMeveGour,Ewin,HeylLammSas,Juan,PeetRoec}.

We assume that, after re-scaling by a constant factor, the probability law of the observations is the same across sensors.
This assumption, although seemingly restrictive, is realistic in view of the fact that the system of sensors is coupled.
The constant then reflects the fact that signal strength can differ in different locations in the system of sensors.
We model the signals through continuous-time It\^o processes.
The advantage of such models is the fact that they can capture complex dependencies in the observations.
For example, an autoregressive process is a special case of the discrete-time equivalent of an Ornstein-Uhlenbeck process, which in turn, is a special case of an It\^o process.
Other special cases of this model include Markovian models, and
linear state-space systems commonly used in vibration-based structural analysis and health monitoring problems \cite{BassAbdeBenv,BassBenvGourMeve,BassMeveGour,Ewin,HeylLammSas,Juan,PeetRoec}. It is important to stress that
the fact that the system of $N$ sensors is coupled makes the probabilistic treatment of the problem equivalent to the one in which all observations become available in one location. The reason is that one integrated information flow is sufficient for describing such a system.

Our objective is to detect the first onset of a signal in such a system. In other words, we wish
to detect the first change point in the system of sensors described above. Our problem is thus concerned with
on-line change detection. We assume that the change points are unknown constants and adopt a min-max approach to
the detection of the minimal one. We set up the problem as a stochastic optimization problem in which the objective
is to minimize an extended Kullback-Leibler distance which serves as a detection delay subject to a lower bound in the
mean time to the first false alarm. So far, of the min-max type of a set-up only the case of an uncoupled system of observations
has been considered and only in the case that the signals across sensors are modeled as independent Brownian motions with constant drifts
\cite{HadjZhanPoor}.\footnote{The situation with correlated Brownian motions is studied in \cite{corr, RobustCUSUM}.} Our problem is also similar to the one discussed in \cite{BayrPoor} which is concerned with the
detection of the minimum of two Poisson disorder times, but in which a Bayesian approach to modeling the change points is taken.
Another interesting filtering approach to change-point detection is developed in a Bayesian framework in \cite{VellClar}.
Recently the case was also considered of change points that propagate in a sensor array \cite{RaghVeer}.
However, in this configuration the propagation of the change points depends on the \underline{unknown} identity of the
first sensor affected and considers a restricted Markovian mechanism of propagation of the change.
For a summary of the latest work on the topic of quickest detection, please consult \cite{PoorHadj}.

In
this paper we consider the case in which the change points can be different and do not propagate
in any specific configuration. The objective is to detect the minimum (i.e., the first) of the change points.
We demonstrate that, in
the situation described above, at least asymptotically, the minimum of $N$
cumulative sums (CUSUMs) \cite{Mous06,Tart05}, is asymptotically optimal in detecting the minimum of the $N$
different change points, as the mean time to the first false alarm tends
to infinity, with respect to an appropriately extended Kullback-Leibler divergence
criterion \cite{Mous04} that incorporates the possibility of $N$
different change points.
The results in this paper generalize those of \cite{HadjZhanPoor} in that the minimum of $N$ CUSUMs is asymptotically optimal not only in the case in which the observations across sensors are independent but also in the case in which they are coupled. In fact in this paper we not only generalize the asymptotic optimality of the $N$-CUSUM stopping rule to a coupled system of observations but also to more general It\^o dynamics rather than the restricted Brownian motion model of observations.

In the next section we formulate the problem and identify the quantities of interest.
In \S 3 we analyze the properties of the proposed $N$-CUSUM stopping rule. In \S 4, we provide asymptotic expansions of the quantities of interest as the threshold parameters of the $N$-CUSUM stopping rule tend to infinity. In \S 5, we demonstrate that the minimum of $N$
CUSUM stopping rules is asymptotically optimal in an extended Kullback-Leibler sense both in the
case of equal and in the case of unequal signal strengths. The proofs of lemmas and propositions that are omitted  can be found in appendices.  Throughout the paper, we denote by $\mathbb{R}=(-\infty,\infty),\,\mathbb{R}_+=[0,\infty)$ and $\ol{\mathbb{R}}_+=[0,\infty]$.

\section{Mathematical formulation}

We sequentially observe the processes
$\{Z_t^{(i)}\}_{t\ge 0}$ for all $i\in\{1,\ldots,N\}$.
In order to formalize this problem we consider filtered probability space $(\Omega, \mathbb{F}, \mathbb{P})$ with filtration $\mathbb{F}=\{\mathcal{F}_t\}_{t\ge0}$ and
$\mathcal{F}_t=\sigma\{(Z_s^{(1)},\ldots, Z_s^{(N)});s\le t\}$.

The processes
$\{Z_t^{(i)}\}_{t\ge 0}$ for all $i\in\{1,\ldots,N\}$ are It\^o processes satisfying:
\begin{eqnarray}
\label{Itodynamics} dZ_t^{(i)} = \left \{
\begin{array}{ll}
dw_t^{(i)} & t \le \tau_i \\
\alpha_t^{(i)}\,dt + dw_t^{(i)}& t > \tau_i,
\end{array}
\right.
\end{eqnarray}
 where $\{w_t^{(i)}\}_{t\ge 0}$'s are independent standard Brownian motions, and $\tau_i's\in\ol{\mathbb{R}}_+$ denote the times of change-points, which are assumed to be deterministic constants but otherwise unknown. Throughout, we will denote the minimum of $\tau_i$'s by,\footnote{As a convention, if $\tau_i=\infty$ for all $i\in\{1,\ldots,N\}$, we let $\tilde{\tau}=\infty$.} 
 \begin{eqnarray}
 \tilde{\tau}:=\min_{1\le i\le N}\{\tau_i\}.\nn
 \end{eqnarray}
  Finally, we assume that $\alpha_t^{(i)}=\alpha_i(t;Z^{(1)}_\cdot,\ldots, Z^{(N)}_\cdot)$ for some predictable real-valued function $\alpha_i$ on $\mathbb{R}_+\times(C(\mathbb{R}_+))^N$, satisfying the Lipschitz condition and linear growth condition (see for example, Theorem 2.9 on page 289 of \cite{KaratzasShreve91}, or Theorem 2.1 on page 375 of \cite{RevuzYor}). 

The case considered in this paper is that in which there exist constants $\{c_i\}_{1\le i\le N}$ such that the drift processes after $\tau_i$'s are proportional to one another. More specifically, without loss of generality, 
  we assume that: 
  
 \begin{assumption}\label{assum1}There exist
 $1=c_1\le |c_2|\le\ldots\le |c_N|$ such that the processes 
 $|\alpha_t^{(1)}|$, $|c_2\alpha_t^{(2)}|$, $\ldots$, $|c_N\alpha_t^{(N)}|$ are modifications of one another under $\mathbb{P}$.
 \end{assumption}

In particular, we examine two cases that are treated separately in the sequel; the first one corresponds to signal-strength symmetry across sensors and the second one corresponds to  across sensors. The former is described mathematically by assuming $|c_i|=1$ for all $i\in\{1,\ldots,N\}$ and the latter, can be captured mathematically by assuming $\max_{2 \le i \le N} |c_i|>1$.

Although Assumption \ref{assum1}  is mathematically necessary for the main results of this paper, in most applications the nature of $\alpha_t^{(i)}$'s is the same across $i$. In fact, in most applications the signal observed across all sensors $i$ is the same but can be felt at different strengths (the constants $c_i$ represent the different strengths) in different sensors. Our formulation treats the general case in which the signal can  arrive at different points in time in different sensors.  

In what follows we describe an example of two sensors each meant to capture two coupled waves that are emitted as one seismic signal after the onset of an earthquake, namely the S wave and the T-wave (see, for instance, the discussion following page 392 in \cite{BassNiki}).
 These waves are not only buried in noise but have the tendency to interact with each other making the detection of their onset particularly difficult and giving obvious rise to a coupled system. These waves in practice may actually have different onsets as discussed in \cite{BassNiki}. Let us now represent by $Z_t^{(1)}$ and $Z_t^{(2)}$ the observations received in the sensors meant to detect the onset of the earthquake. Before the onset of the earthquake both observations are simply noise. After the onset of the earthquake a signal will be emitted but may be received at different points in time in the different sensors. The models used to capture signals are often autoregressive \cite{BassNiki} whose continuous-time equivalents are Ornstein-Uhlenbeck processes. Thus we need to allow for $\alpha_t^{(i)}$ to depend on $Z_t^{(i)}$. Of course more general models would be necessary in order to capture dependence on past oservations as well. This is the reason that in our model \eqref{Itodynamics}, $\alpha_t^{(i)}$ are generalized to depend on the whole of the path of the observations $Z_t^{(i)}$ up to time $t$. Moreover, a stronger intensity in one of the two waves is known to result in a stronger intensity in the other \cite{BassNiki}. This could then be modeled by allowing both $\alpha_t^{(1)}$ and $\alpha_t^{(2)}$ to have a positive dependence on $Z_t^{(2)}$ and $Z_t^{(1)}$ respectively (note that this dependence does not have to be linear and could in general cases depend on prior observations too-not just the ones at the current time). Even though the two waves are different, they do give rise to the same signal which may be felt with a different strength in each of the sensors and thus the restriction $c_2 \alpha_t^{(2)}=\alpha_t^{(1)}$ with $c_2>0$, although may appear restrictive, is in fact not far from modeling reality. We notice of course that in this example we do not have to impose the absolute value condition as both $\alpha_t^{(1)}$ and $\alpha_t^{(2)}$ will have the same sign. 
 
  Another application is that of a vibrating mechanical system where each sensor may be monitoring a different location on the structure. The signal causing the vibration can be felt at different times in different locations and also with different strengths $c_i$. Let $Z_t^{(i)}$ represent the observations at sensor $i$. At first there is no signal and the sensors only observe noise. After the onset of the signal, which may be felt at different times in different sensors, a signal arrives. A signal causing a vibration in one location may also affect the observations received in another. This leads to feedback \cite{BassAbdeBenv}-\cite{BassMeveGour}, which can then be modeled by allowing $\alpha_t^{(i)}$ to depend on the observations received at all sensors. Also the signal may be more pronounced in one sensor than in another depending on the location of the sensor. This can be captured by allowing the $c_i$'s to be different across sensors. In fact, by allowing the $c_i$'s to have different signs we are able to capture the situation in which the sensors are hit by the signal from opposite directions.

The filtration $\mathbb{F}$ is the filtration generated by the observations received by all sensors. Thus by requiring that
$\alpha_t^{(i)}$ be $\mathcal{F}_t$-measurable for all $i$ , we have managed to capture the coupled nature of the system. In particular, in the special case in which, say, $\alpha_t^{(1)}=-r\sum_{i=1}^N Z_t^{(i)}$, (\ref{Itodynamics}) describes a process that displays an autoregressive (or its continuous equivalent \cite{Novi}) behavior in $\{Z_t^{(1)}\}_{t \ge 0}$, while still being coupled with the observations received by the other sensors. More specifically, the magnitude of each increment of the process $\{Z_t^{(1)}\}_{t \ge 0}$ at each instant $t$ is not only affected by $Z_t^{(1)}$ but also by $Z_t^{(i)}$, $i=2,\ldots,N$, the observations at sensors $2,\ldots,N$. This couples the observations received in sensor $1$ with those received in sensors $2,\ldots,N$ at each instant $t$ and results in a system of interdependent sensors. We notice that the special case described above can also be written in the form of a linear state-space model as follows:
\begin{eqnarray}\label{example}
dZ_t^{(i)}=-r\idc{t>\tau_i}\bigg(\sum_{j=1}^NZ_t^{(j)}\bigg)dt+dw_t^{(i)},\,\,\,i\in\{1,\ldots, N\}.
\end{eqnarray}
Autoregressive models and, more generally, linear state space models have been used to capture seismic signals, navigation systems, vibrating mechanical systems, etc \cite{BassNiki}.
The generality of \eqref{Itodynamics} however is much greater than the special case described above. This is seen in the fact that $\alpha_t^{(i)}$ at each instant $t$ can depend on the totality of the observed paths of each of the signals received up to time $t$. Notice that this linear model clearly satisfies both the Lipschitz condition and the linear growth condition.  To incorporate the different signal strength in each of the sensor observations, we could update the present model in (\ref{example}) by having different constants in each row of the second $N \times N$ matrix appearing in the equation of the linear system above.

To formulate our problem mathematically, we will work with 
 the canonical space $(C(\mathbb{R}_+))^N$ endowed with the filtration $\mathbb{B}=\{\mathcal{B}_t\}_{t\ge0}$, where $\mathcal{B}_t=\sigma(\omega(s); s\le t)$ is the Borel $\sigma$-algebra generated by the the joint coordinate mapping $\omega=(\omega_1,\ldots,\omega_N)$: 
 for  $Z_s^{(i)}(\omega)=\omega_i(s)$, $s\ge0$, $i\in\{1,\ldots,N\}$.  
 
 We introduce a family of probability measures on $((C(\mathbb{R}_+))^N,\mathbb{B})$: $\{P_{\mathbf{t}}\}_{\mathbf{t}\in (\overline{\mathbb{R}}_+)^N}$, where $P_{\mathbf{t}}$ for a fixed $\mathbf{t}=(\tau_1,\ldots,\tau_N)\in(\ol{\mathbb{R}}_+)^N$
corresponds to the measure generated on $(C(\mathbb{R}_+))^N$ by the $N$-dimensional process
$\{(Z_t^{(1)},\ldots,Z_t^{(N)})\}_{t\ge0}$ when the change in $Z_\cdot^{(i)}$  occurs at time point $\tau_i$, for all $i\in\{1,\ldots,N\}$. In particular, the case that $\tau_i=\infty$ for all $i\in\{1,\ldots,N\}$ corresponds to no change regime, where $(Z_t^{(1)},\ldots,Z_t^{(N)})_{t\ge0}$ is a $N$-dimensional Brownian motion; and the measure  generated on $(C(\mathbb{R}_+))^N$ is the 
the $N$-dimensional Wiener measure, which we denote by $P_{\mathbf{t}_\infty}$ for $\mathbf{t}_\infty:=(\infty,\ldots,\infty)$.   Now we can restate Assumption \ref{assum1} as the following: 

\begin{assumption}\label{assum1p}
There exist
 $1=c_1\le |c_2|\le\ldots\le |c_N|$ such that the processes 
 $|\alpha_1(t;\omega)|$, $|c_2\alpha_2(t;\omega)|$, $\ldots$, $|c_N\alpha_N(t;\omega)|$ are modifications of one another under $P_{\mathbf{t}}$ for all $\mathbf{t}\in(\ol{\mathbb{R}}_+)^N$.
\end{assumption}

We will also assume the  drift functions $\alpha_i(\cdot;\cdot)$ satisfy the Novikov condition:
\begin{assumption}For all $t\in\mathbb{R}_+$, we have
\begin{equation}\label{novikov}
E_{\mathbf{t}_\infty}\bigg\{\exp\bigg(\frac{1}{2}\int_0^t\sum_{i=1}^N(\alpha_i(s;\omega))^2ds\bigg)\bigg\}<\infty.
\end{equation}
\end{assumption}
By Theorem 1.10 on page 371 of \cite{RevuzYor}, we know that, if the change-points satisfy $0\le\tau_1<\tau_2<\min_{2<i\le N}\{\tau_i\}$, we have
\begin{align}
\frac{dP_{(\tau_1,\tau_2,\ldots,\tau_N)}}{dP_{(\infty,\infty,\ldots,\infty)}}\bigg|_{\mathcal{B}_t}=&\frac{dP_{(\tau_1,\infty,\ldots,\infty)}}{dP_{(\infty,\infty,\ldots,\infty)}}\bigg|_{\mathcal{B}_t}=\exp(u_1(t)-u_1(t\wedge{\tau_1})),\,\,\,\forall t\in[0,\tau_2],\,\,\,\text{and}\nn\\
\frac{dP_{(\tau_1,\tau_2,\tau_3,\ldots,\tau_N)}}{dP_{(\tau_1,\infty,\infty,\ldots,\infty)}}\bigg|_{\mathcal{B}_t}=&\frac{dP_{(\tau_1,\tau_2,\infty,\ldots,\infty)}}{dP_{(\tau_1,\infty,\infty,\ldots,\infty)}}\bigg|_{\mathcal{B}_t}=\exp(u_2(t)-u_2(\tau_2)),\,\,\,\forall t\in[\tau_2,\min_{2<i\le N}\{\tau_i\}],\nn
\end{align}
and so on, where, for all $ i\in\{1,\ldots, N\}$,
\begin{equation}\label{u process}u_i(t):=\int_0^t\alpha_i(s;\omega)d\omega_i(s)-\int_0^t\frac{1}{2}(\alpha_i(s;\omega))^2ds.\end{equation}
In general, for all $\mathbf{t}=(\tau_1,\ldots,\tau_N)\in(\ol{\mathbb{R}}_+)^N$ we have that, 
\begin{equation}\label{RNDeri}
\frac{dP_{\mathbf{t}}}{dP_{\mathbf{t}_\infty}}\bigg|_{\mathcal{B}_t}=\exp\bigg(\sum_{i=1}^N(u_i(t)-u_i(t\wedge{\tau_i}))\bigg),\,\,\,\forall t\in\mathbb{R}_+.
\end{equation}

Our objective is to find an $\mathbb{B}$-stopping rule $T$ that balances the
trade-off between a small detection delay subject to a lower bound
on the mean-time to the first false alarm and will ultimately detect
$\tilde\tau$. 

To this effect, we propose a generalization of the $J_{KL}$ of \cite{Mous04}, namely
\begin{eqnarray} \label{JKL} J_{KL}^{(N)}(T) &:= & \mathop{\sup_{\mathbf{t}=(\tau_1,\ldots,\tau_N)\in(\ol{\mathbb{R}}_+)^N}}_{\tilde{\tau}<\infty} \mathrm{essup}~E_{\mathbf{t}}\bigg\{\idc{T>\tilde{\tau}} \int_{\tilde{\tau}}^T \frac{1}{2}(\alpha_1(s;\omega))^2ds|\mathcal{B}_{\tilde{\tau}}\bigg\},
\end{eqnarray}
 where the supremum over $\mathbf{t}=(\tau_1, \ldots, \tau_N)$ is taken over
the set in which $\tilde{\tau}<\infty$. That is, we consider the
worst detection delay over all possible realizations of paths of the $N$-tuple of
stochastic processes $\{(Z_t^{(1)},\ldots,Z_t^{(N)})\}_{t\ge0}$ up to
$\tilde{\tau}$ and then consider the worst detection delay over all
possible $N$-tuples $\mathbf{t}=(\tau_1,\ldots,\tau_N)$ over a set in which at least one of
them is forced to take a finite value. This is because $T$ is a stopping rule meant
to detect the minimum of the $N$ change points and therefore if one of the $N$
processes undergoes a regime change, any unit of time by which $T$ delays in
reacting, should be counted towards the detection delay.
This gives rise to the
following stochastic optimization problem:
\begin{equation}
  \begin{array}{c}
    \inf_{T\in\mathbb{B}_\gamma}J_{KL}^{(N)}(T)\\
    ~\text{with }\mathbb{B}_\gamma=\{\mathbb{B}\text{-stopping rule }T~:~E_{\mathbf{t}_\infty}\{\int_0^T\frac{1}{2}(c_N^{-1}\alpha_1(s;\omega))^2ds\}\geq\gamma\}
\end{array}
    \label{eqnproblemKLG},\end{equation}
The constraint in the above optimization problem is a measure on the first time to the first false alarm. To see this, one can recognize the functional $\int_0^t \frac{1}{2}(\alpha_s^{(i)})^2 ds$ as a measure of accumulation of energy of the signal \cite{LiptShir} up to time $t$. The above constraint involves the expected value of this functional under the regime of no signal up to the time of stopping and is thus a measure of a mean time to the first false alarm \cite{LiptShir, Mous04,PoorHadj}. Similarly, the criterion \eqref{JKL} captures the accumulation of energy of the signal after the first instance at which a signal becomes available. It is possible to motivate this criterion through a Kullback-Leibler distance.
To be more specific, for any bounded stopping rule $T$ such that 
\[c_i^2\int_0^T\alpha_i(s;\omega)^2ds=\int_0^T\alpha_1(s;\omega)^2ds<\infty,\,\,\,P_{\mathbf{t}}\text{-a.s.},\,\,\,\forall\mathbf{t}=(\tau_1,\ldots,\tau_N)\in(\overline{\mathbb{R}}_+)^N,\]
we can apply \eqref{RNDeri} to obtain that
\begin{align}
E_{\mathbf{t}}\bigg\{\log\frac{dP_{\mathbf{t}}}{dP_{\mathbf{t}_\infty}}\bigg|_{\mathcal{B}_T}
 |\mathcal{B}_{\tilde{\tau}}\bigg\} &=E_{\mathbf{t}}\bigg\{\sum_{i=1}^N(u_i(T)-u_i(T\wedge\tau_i))|\mathcal{B}_{\tilde{\tau}}\bigg\}\nn\\
 &=\sum_{i=1}^{N} E_{\mathbf{t}}\bigg\{\idc{T>\tau_i}\bigg(\int_{\tau_i}^T\alpha_i(s;\omega)d\omega_i(s)-\int_{\tau_i}^T\frac{1}{2}(\alpha_i(s;\omega))^2ds\bigg)|\mathcal{B}_{\tilde{\tau}}\bigg\}\nn\\
 &= \sum_{i=1}^N E_{\mathbf{t}}\bigg\{\idc{T>\tau_i}\int_{\tau_i}^T\frac{1}{2}(\alpha_i(s;\omega))^2ds|\mathcal{B}_{\tilde{\tau}}\bigg\},\nn
\end{align}
giving rise to the criterion
\begin{eqnarray}
\mathcal{J}(T):= \mathop{\sup_{\mathbf{t}=(\tau_1,\ldots,\tau_N)\in(\ol{\mathbf{R}}_+)^N}}_{\tilde{\tau}<\infty} \text{essup} \sum_{i=1}^N E_{\mathbf{t}}\bigg\{\idc{T>\tau_i}\int_{\tau_i}^T\frac{1}{2}(\alpha_i(s;\omega))^2ds|\mathcal{B}_{\tilde{\tau}}\bigg\},\nn
\end{eqnarray}
which can be upper bounded by the criterion in \eqref{JKL} as follows:
\begin{align}
\label{UB}
\mathcal{J}(T) \le &   \mathop{\sup_{\mathbf{t}=(\tau_1,\ldots,\tau_N)\in(\ol{\mathbf{R}}_+)^N}}_{\tilde{\tau}<\infty} \sum_{i=1}^N \text{essup} E_{\mathbf{t}}\bigg\{\idc{T>\tau_i}\int_{\tau_i}^T\frac{1}{2}(\alpha_i(s;\omega))^2ds|\mathcal{B}_{\tilde{\tau}}\bigg\}\\ 
=&  \mathop{\sup_{\mathbf{t}=(\tau_1,\ldots,\tau_N)\in(\ol{\mathbf{R}}_+)^N}}_{\tilde{\tau}<\infty}  \sum_{i=1}^N \text{essup} E_{\mathbf{t}}\bigg\{\idc{T>\tau_i}(c_i)^{-2}\int_{\tau_i}^T\frac{1}{2}(\alpha_1(s;\omega))^2ds |\mathcal{B}_{\tilde{\tau}}\bigg\} \nn\\ 
\nonumber  \le & N    \mathop{\sup_{\mathbf{t}=(\tau_1,\ldots,\tau_N)\in(\ol{\mathbf{R}}_+)^N}}_{\tilde{\tau}<\infty} \text{essup} E_{\mathbf{t}}\bigg\{\idc{T>\tau_i}\int_{\tau_i}^T\frac{1}{2}(\alpha_1(s;\omega))^2ds |\mathcal{B}_{\tilde{\tau}}\bigg\} \\   \le & N J_{KL}^{(N)}(T).\nn
\end{align}

Similarly, for the false alarms we notice that
\begin{align}
\label{LB}
E_{\mathbf{t}_\infty}\bigg\{\sum_{i=1}^N \int_0^T\frac{1}{2} (\alpha_i(s;\omega))^2ds\bigg\} \ge& N \min_{1\le i\le N} E_{\mathbf{t}_\infty}\bigg\{\int_0^T \frac{1}{2}(\alpha_i(s;\omega))^2ds\bigg\}\\=&NE_{\mathbf{t}_\infty}\bigg\{\int_0^T \frac{1}{2}(c_N^{-1}\alpha_1(s;\omega))^2ds\bigg\}.\nn\end{align}
 Dividing the upper bound in (\ref{UB}) and the lower bound in (\ref{LB}) by the number of sensors $N$ results in the optimization problem given in \eqref{eqnproblemKLG}.

Inspired by the optimality of equalizer rules as seen in \cite{HadjZhanPoor, corr}, we will examine the performance under \eqref{JKL}, of an (almost) equalizer rule $T$ which reacts at exactly the same time regardless of which change takes place first. To be more specific, let us define 
\begin{eqnarray}\label{JjN}
J_j^{(N)}(T) & := & \mathop{\sup_{\mathbf{t}=(\tau_1,\ldots,\tau_N)\in(\ol{\mathbf{R}}_+)^N}}_{\tau_j=\tilde{\tau}<\infty} \mathrm{essup}E_{\mathbf{t}}\bigg\{\idc{T>\tau_j}\int_{\tau_j}^T \frac{1}{2}
 (\alpha_1(s;\omega))^2ds|\mathcal{B}_{\tau_j}\bigg\},
\end{eqnarray}
for all $j\in\{1,\ldots,N\}$. That is, $J_j^{(N)}(T)$ is the detection delay
of the stopping rule $T$ when $\tau_j=\tilde{\tau}$ is the first change-point.
Then for all stopping rules $T\in\mathbb{B}_\gamma$,
\begin{eqnarray}\label{decomp}
J_{KL}^{(N)}(T)=\max\left\{J_1^{(N)}(T),J_2^{(N)}(T),\ldots,J_N^{(N)}(T)\right\}.
\end{eqnarray}
We will restrict our attention to stopping rules $T\in\mathbb{B}_\gamma$ that achieve the
false alarm constraint with equality  \cite{Mous86,Mous04}:
\begin{equation}\label{FAeq}
E_{\mathbf{t}_\infty}\bigg\{\int_0^T\frac{1}{2}(c_N^{-1}\alpha_1(s;\omega))^2ds\bigg\}=\gamma,
\end{equation}
and the ``equalizer rule'' condition:
\begin{equation}\label{EQR}
J_j^{(N)}(T)=J_1^{(N)}(T)+o(1),\,\,\,\forall j\in\{1,2,\ldots,N\},
\end{equation}
as $\gamma\to\infty$. Clearly, from \eqref{decomp} we know that, for such stopping rules we may simplify the performance index $J_{KL}^{(N)}(T)$ as 
\[J_{KL}^{(N)}(T)=J_1^{(N)}(T)+o(1),\]
as $\gamma\to\infty$. While, given the complexity due to the dimensionality of the problem, it seems a formidable task to determine the exactly optimal stopping rule for \eqref{eqnproblemKLG}, asymptotic expansions like \eqref{EQR} will make it possible   to find an asymptotically optimal stopping rule for our problem \eqref{eqnproblemKLG}.

To construct such a candidate stopping for problem \eqref{eqnproblemKLG},  recall that in the case of $N=1$,  
the drift, $\alpha_1(t;\omega_1)$, is measurable with respect to the filtration 
generated by the coordinate mapping $Z_t^{(1)}(\omega)=\omega_1(t)$, $t\ge0$, which we  denote by  $\mathbb{G}^{(1)}=\{\mathcal{G}_t^{(1)}\}_{t\ge0}$.
Then it is known from \cite{Mous04} that the problem in \eqref{eqnproblemKLG}, when $N=1$, is solved by the so-called CUSUM stopping rule: 
\begin{eqnarray}
\label{CUSUMIto}
S_\nu^{(1)}  =  \inf\{t \ge 0~:~y(t)=\nu\},\,\,\,\text{with}\,\,\,
 y(t) =  \sup_{0\le s \le t} \log\frac{dP_{s}}{dP_\infty}\bigg|_{\mathcal{G}_t^{(1)}}.
\end{eqnarray}
where $P_{s}$ is 
the measures generated by $\{Z_t^{1}\}_{t\ge0}$ on $C(\mathbb{R}_+)$, when the change-point $\tau_1=s$; and $P_\infty$ is the Wiener measure on $C(\mathbb{R}_+)$. 
The threshold $\nu>0$ in \eqref{CUSUMIto} is chosen so that $E_\infty\{\int_0^{S_{\nu}^{(1)}}\frac{1}{2}(\alpha(s;\omega_1))^2ds\}=\gamma$. Furthermore, with 
\begin{equation}\label{g} g(\nu)\,:=\,\textrm{e}^\nu-\nu-1,\end{equation}
 the threshold $\nu$ is the unique positive solution to $g(\nu)=\gamma$, and the optimal detection delay is given by\begin{eqnarray}
\nn J_{KL}^{(1)}(S_{\nu}^{(1)})=E_0\bigg\{\int_0^{S_{\nu}^{(1)}}\frac{1}{2}(\alpha_1(s;\omega_1))^2ds\bigg\}=
g(-\nu).
\end{eqnarray}

In the proof in  \cite{Mous04}  that proves the optimality of the CUSUM stopping rule \eqref{CUSUMIto} to the  one-dimensional equivalent of \eqref{eqnproblemKLG}, one vital assumption  is necessary for the finiteness (see Theorem 1 of \cite{Mous04}) of the CUSUM stopping rule \eqref{CUSUMIto} and the martingale properties of stochastic integrals, i.e., for $\tau_1\in\{0,\infty\}$, and any $t\in\mathbb{R}_+$,
\begin{eqnarray}\nonumber
 P_{\tau_1}\bigg(\int_0^\infty \frac{1}{2}(\alpha_1(s;\omega_1))^2
ds=\infty\bigg) & = &P_{\tau_1}\bigg(\int_0^t \frac{1}{2}(\alpha_1(s;\omega_1))^2
ds<\infty\bigg) =1.
\end{eqnarray}
Intuitively, a physical interpretation of the above assumption is that the
signal over the whole positive real line has infinite energy, and thus the CUSUM stopping rule $S_\nu^{(1)}$ in \eqref{CUSUMIto} will goes off in finite time with probability 1. 

In this work, we will thus assume analogous conditions:  
\begin{assumption}
For all $\tau_i\in\{0,\infty\}$, $i\in\{1,\ldots,N\}$, and any $t\in\mathbb{R}_+$,\begin{eqnarray}
\label{Energyi}
P_{(\tau_1,\ldots,\tau_N)}\left(\int_0^\infty \frac{1}{2}(\alpha_1(s;\omega))^2 ds=\infty\right) & = &1,\\
P_{(\tau_1,\ldots,\tau_N)}\left(\int_0^t\frac{1}{2}(\alpha_1(s;\omega))^2 ds<\infty\right)&=&1.\label{Energyi1}
\end{eqnarray}
\end{assumption}
We comment that  \eqref{Energyi1} is implied by \eqref{novikov}, since the measures $P_{\tau_1,\ldots,\tau_N}|_{\mathcal{B}_t}$ are all equivalent under \eqref{novikov}.

A class of functions $\alpha_i(\cdot;\cdot)$'s that satisfy \eqref{novikov}, \eqref{Energyi} and \eqref{Energyi1} are the uniformly bounded Lipschitz continuous functions that are also uniformly bounded away from 0. It can be easily shown that linear state space models of the form \eqref{example} also satisfy \eqref{novikov}, \eqref{Energyi} and \eqref{Energyi1}.
 
\begin{lemma}\label{linear}
If $\alpha_1(t;\omega)=\sum_{i=1}^N\beta_i\omega_i(t)$, and $\alpha_j(t;\omega)=c_j\alpha_1(t;\omega)$ for some real constants $\beta_i,c_j$ such that $\sum_{i=1}^N\beta_i^2>0$ and $|c_N|\ge\ldots \ge|c_2|\ge c_1=1$, then \eqref{novikov}, \eqref{Energyi} and \eqref{Energyi1} hold.
\end{lemma}

 The optimality of the CUSUM
stopping rule in the presence of only one observation process suggests that a CUSUM
type of stopping rule might display similar optimality properties in the case of
multiple observation processes.
In particular, an intuitively appealing rule, when
the detection of $\tilde{\tau}=\min_i\{\tau_i\}$ is of interest, is the $N$-dimensional variant of the CUSUM stopping rule $T_{\hbar}=\min_i\{T_{h_i}^{(i)}\}$,
where $T_{h_i}^{(i)}$ is the CUSUM stopping rule for the
process $\{Z_t^{(i)}\}_{t \ge 0}$ for $i\in\{1,\ldots,N\}$. That is, we use what is known
as a multi-chart CUSUM stopping rule \cite{Mous06}, which can be written as
\begin{eqnarray}
\label{CUSUMmultichart} T_{\hbar} & = & \inf\bigg\{t \ge 0~:~
\max\bigg\{\frac{y_1(t)}{h_1},\ldots,\frac{y_N(t)}{h_N}\bigg\} \ge 1\bigg\}.
\end{eqnarray}
Here for any $i\in\{1,\ldots,N\}$, $h_i>0$ is the threshold for process $\{y_i(t)\}_{t\ge0}$, and 
$\{y_i(t)\}_{t\ge0}$ is defined as (see \eqref{u process}-\eqref{RNDeri}):
\begin{align}
\label{yti} 
y_i(t)& :=  \sup_{0 \le s \le t}\log
\frac{dP_{(\infty, \ldots, s, \ldots, \infty)}}{dP_{\mathbf{t}_\infty}}\bigg|_{\mathcal{B}_t} =\sup_{0\le s\le t}\{u_i(t)-u_i(s)\}=u_i(t)-m_i(t)\\
&=\int_0^t\alpha_i(s;\omega)d\omega_i(s)-\frac{1}{2}\int_0^t(\alpha_i(s;\omega))^2ds-m_i(t),\nn
\end{align}
where $P_{(\infty, \ldots, s, \ldots, \infty)}$ is corresponding to the change-points $\tau_j=\infty$ for all $j\neq i$ and $\tau_i=s$,
$u_i(t)$ is given in \eqref{u process} and  $m_i(t):=\inf_{0\le s\le t}u_i(s)$.
For any fixed large $\gamma>0$, the $N$-dimensional threshold vector $\hbar=(h_1,\ldots,h_N)\in(\mathbb{R}_+)^N$ is to be determined so that both \eqref{FAeq} and \eqref{EQR} hold.

We will give an analytical characterization of the stopping rule $T_\hbar$ in the next section.

\section{Multi-chart CUSUM stopping rule $T_\hbar$}
In this section, we study several property of the multi-chart CUSUM stopping rule $T_\hbar$ defined in \eqref{CUSUMmultichart}. First, we show its a.s. finiteness, and compute the Kulback-Leibler detection and false-alarm divergence of $T_\hbar$. We then establish an upper bound for  $\inf_{T\in\mathbb{B}_\gamma}J_{KL}^{(N)}(T)$, which  will be used in Section 5 for asymptotic analysis. Throughout this section, we will denote by $Y(t):=(y_1(t),\ldots,y_N(t))$ for $y_i(t)$'s defined in \eqref{yti}.

Let us begin by introducing a class of functions that will be used  in the main results. For any fixed vectors $\mathcal{S}=(\mathcal{S}_1,\ldots,\mathcal{S}_N)\in\{\pm1\}^N$ and $\hbar\in(\mathbb{R}_+)^N$, we define $f_{\mathcal{S},\hbar}(\cdot)$ as a $C^2$ function on $D_\hbar:=[0,h_1]\times\ldots\times[0,h_N]$, satisfying the following PDE:\begin{subnumcases}{}\label{pde1}\sum_{i=1}^N \frac{1}{c_i^{2}}\frac{\partial^2 f_{\mathcal{S},\hbar}}{\partial y_i^2}+\sum_{i=1}^N\frac{\mathcal{S}_i}{c_i^2}\frac{\partial f_{\mathcal{S},\hbar}}{\partial{y_i}}=-1,\,\,\,\forall (y_1,\ldots,y_N)\in D_\hbar,\\
\label{cond1}\frac{\partial f_{\mathcal{S},\hbar}}{\partial y_i}\bigg|_{y_i=0}=f_{\mathcal{S},\hbar}|_{y_i=h_i}=0.
\end{subnumcases}

Using  Feynman-Kac representation of $f_{\mathcal{S},\hbar}$, we obtain the following useful monotonicity properties:
\begin{lemma}\label{pde lemma}
For any $\mathcal{S}=(\mathcal{S}_1,\ldots,\mathcal{S}_N)\in\{\pm1\}^N$, the function $f_{\mathcal{S},\hbar}$  defined in \eqref{pde1}-\eqref{cond1} are decreasing in $y_i$ for each $i$. If $\mathcal{S}'=(\mathcal{S}'_1,\ldots\mathcal{S}_N')\in\{\pm1\}^N$ such that $\mathcal{S}_i'\le\mathcal{S}_i$ for any $i=1,\ldots, N$, then for all $(y_1,\ldots,y_N)\in D_\hbar$,
\begin{equation}
0\le f_{\mathcal{S},\hbar}(y_1,\ldots,y_N)\le f_{\mathcal{S}',\hbar}(y_1,\ldots,y_N)\le f_{\mathcal{S}',\hbar}(0,\ldots,0)<\infty.
\end{equation}
Moreover, $f_{\mathcal{S},\hbar}(0,\ldots,0)$ is strictly increasing in $h_i$ for each $i$.
\end{lemma}

The following technical lemma will also be useful for localization argument later.
\begin{lemma}\label{finite lemma}
For any fixed $\tau,\tau'\in\mathbb{R}_+$, and a fixed $\mathbf{t}=(\tau_1,\ldots,\tau_N)\in(\ol{\mathbb{R}}_+)^N$, we have
\[P_{\mathbf{t}}\bigg(\int_\tau^\infty\frac{1}{2}(\alpha_1(s;\omega))^2ds=\infty|\mathcal{B}_{\tau'}\bigg)=1.\]
\end{lemma}

\begin{theorem}\label{exp cusum}
For any fixed vector $\hbar\in \mathbb{R}_+^N, \mathbf{t}=(\tau_1,\ldots,\tau_N)\in\ol{\mathbb{R}}_+^N$ such that  $\tilde{\tau}=\min_i\{\tau_i\}<\infty$,  the multi-chart CUSUM stopping rule $T_\hbar$ is a.s. finite: \begin{eqnarray}
\label{finite1}P_{\mathbf{t}}(T_\hbar<\infty|\mathcal{B}_{\tilde{\tau}})&=&1,\,\,\,P_{\mathbf{t}}\text{-a.s.}
\end{eqnarray}
Define  $\mathcal{S}(\mathbf{t};s):=(\varphi(\tau_1,s),\ldots,\varphi(\tau_N,s))\in\{\pm1\}^N$ with $\varphi(\tau,s):=\idc{\tau< s}-\idc{\tau\ge s}$ for $\tau\in\overline{\mathbb{R}}_+$ and $s\in\mathbb{R}_+$. Let us relabel (without repetition)  the change-points $\tau_1,\ldots,\tau_N$ in increasing order:  for some $l\in\{1,\ldots,N-1\}$,
\[\tilde{\tau}=\tau_{(1)}<\tau_{(2)}<\ldots<\tau_{(l)}<\infty=\tau_{(l+1)}.\]
Then we have
\begin{multline}
\label{KL1}E_{\mathbf{t}}\bigg\{\idc{T_\hbar\ge\tilde{\tau}}\int_{\tilde{\tau}}^{T_\hbar}\frac{1}{2}(\alpha_1(s;\omega))^2ds|\mathcal{B}_{\tilde{\tau}}\bigg\}\\
=\idc{T_\hbar\ge\tilde{\tau}}\sum_{j=1}^{l}E_{\mathbf{t}}\{[f_{\mathcal{S}(\mathbf{t};\tau_{(j)}),\hbar}(Y({\tau}_{(j)}\wedge T_\hbar))-f_{\mathcal{S}(\mathbf{t};\tau_{(j)}),\hbar}(Y({\tau}_{(j+1)}\wedge T_\hbar))]|\mathcal{B}_{\tilde{\tau}}\}.
\end{multline}
Similarly, for a fixed $t\in\mathbb{R}_+$, we have 
\begin{eqnarray}
\label{finite2}P_{\mathbf{t_\infty}}(T_\hbar<\infty|\mathcal{B}_{{t}})&=&1,\,\,\,P_{\mathbf{t}_\infty}\text{-a.s.}\\
\,\,\,\label{KL2}E_{\mathbf{t}_\infty}\bigg\{\idc{T_\hbar\ge t}\int_{{t}}^{T_\hbar}\frac{1}{2}(\alpha_1(s;\omega))^2ds|\mathcal{B}_{t}\bigg\}&=&f_{\mathcal{S}(\mathbf{t}_\infty;t),\hbar}(y_1(t),\ldots,y_N(t))\idc{T_\hbar\ge t}.
\end{eqnarray}

\end{theorem}

 \begin{proof}
 From Lemma \ref{finite lemma}, the stopping rules defined as
\[t_n:=\inf\left\{t\ge\tilde{\tau}\,:\,\int_{\tilde{\tau}}^{t}\frac{1}{2}(\alpha_1(s;\omega))^2ds\ge n\right\},\,\,\,n\ge1,\]
are $P_{\mathbf{t}}$-a.s. finite, and so is $T_\hbar^n:=T_\hbar\wedge t_n$.

Let us apply It\^o's lemma to $\{f_{\mathcal{S}(\mathbf{t};\tau_{(j)}),\hbar}(Y(t))\}_{\tau_{(j)}\wedge T_\hbar^n\le t< \tau_{(j+1)}\wedge T_\hbar^n}$ for any $j\in\{1,\ldots,l\}$. Using \eqref{pde1} we have that,
\begin{align}
\label{Ito's expansion}
&E_{\mathbf{t}}\{f_{\mathcal{S}(\mathbf{t};\tau_{(j)}),\hbar}(Y(\tau_{(j+1)}\wedge T_\hbar^n))-f_{\mathcal{S}(\mathbf{t};\tau_{(j)}),\hbar}(Y(\tau_{(j)}\wedge T_\hbar^n))|\mathcal{B}_{\tilde{\tau}}\}\idc{T_\hbar^n\ge\tilde{\tau}}\\
=&E_{\mathbf{t}}\bigg\{\int_{\tau_{(j)}\wedge T_\hbar^n}^{\tau_{(j+1)}\wedge T_\hbar^n}\sum_{i=1}^N[\frac{\partial f_{\mathcal{S}(\mathbf{t};\tau_{(j)}),\hbar}}{\partial y_i}dy_i(s)+\frac{1}{2}\frac{\partial^2 f_{\mathcal{S}(\mathbf{t};\tau_{(j)}),\hbar}}{\partial y_i^2}(\alpha_i(s;\omega))^2ds]|\mathcal{B}_{\tilde{\tau}}\bigg\}\idc{T_\hbar^n\ge\tilde{\tau}}.\nn
\end{align}
From \eqref{Itodynamics}, \eqref{u process} and \eqref{yti} we know that, for all $i\in\{1,\ldots,N\}$, $dy_i(t)=du_i(t)-dm_i(t)$ and 
\begin{equation}du_i(s)=\alpha_i(s;\omega)dB_i(s)+\frac{1}{2}\varphi(\tau_i,s)(\alpha_i(s;\omega))^2ds,\nn\end{equation}
where $B_i(t)=\omega_i(t)-\idc{t>\tau_i}\int_{\tau_i}^t\alpha_i(s;\omega)ds$, $t\ge0$, is a $\mathbb{B}$-standard Brownian motion under $P_{\mathbf{t}}$. On the other hand, from Assumption \ref{assum1p}, we have for all $i\in\{1,\ldots,N\}$ that
\[du_i(s)=\alpha_i(s;\omega)dB_i(s)+\frac{1}{2}\varphi(\tau_i,s)(c_i^{-1}\alpha_1(s;\omega))^2ds.\]
By substituting the above formula into 
\eqref{Ito's expansion}, we obtain 
\begin{align}\label{BB}
&E_{\mathbf{t}}\bigg\{\int_{\tau_{(j)}\wedge T_\hbar^n}^{\tau_{(j+1)}\wedge T_\hbar^n}[\frac{1}{2}\sum_{i=1}^N(\frac{\varphi(\tau_i,\tau_{(j)})}{c_i^2}\frac{\partial f_{\mathcal{S}(\mathbf{t};\tau_{(j)}),\hbar}}{\partial y_i}+\frac{1}{c_i^2}\frac{\partial^2 f_{\mathcal{S}(\mathbf{t};\tau_{(j)}),\hbar}}{\partial y_i^2})(\alpha_1(s;\omega))^2ds\\
+&\sum_{i=1}^N(\frac{\partial f_{\mathcal{S}(\mathbf{t};\tau_{(j)}),\hbar}}{\partial y_i}\alpha_i(s;\omega)dB_i(s)-\frac{\partial f_{\mathcal{S}(\mathbf{t};\tau_{(j)}),\hbar}}{\partial y_i}dm_i(s))]|\mathcal{B}_{\tilde{\tau}}\bigg\}\idc{T_\hbar^n\ge\tilde{\tau}}\nn\\
=&E_{\mathbf{t}}\bigg\{\int_{\tau_{(j)}\wedge T_\hbar^n}^{\tau_{(j+1)}\wedge T_\hbar^n}[\frac{1}{2}(\alpha_1(s;\omega))^2ds\nn\\
+&\sum_{i=1}^N(\frac{\partial f_{\mathcal{S}(\mathbf{t};\tau_{(j)}),\hbar}}{\partial y_i}\alpha_i(s;\omega)dB_i(s)-\frac{\partial f_{\mathcal{S}(\mathbf{t};\tau_{(j)}),\hbar}}{\partial y_i}dm_i(s))]|\mathcal{B}_{\tilde{\tau}}\bigg\}\idc{T_\hbar^n\ge\tilde{\tau}},\nn\end{align}
where we used \eqref{pde1} in the equality.
Notice that on the event $\{T_\hbar^n<\tau_{(j)}\}$, the integral inside the above expectation is 0; and on the event  $\{T_\hbar^n\ge\tilde{\tau}\}\cap\{T_\hbar^n\ge\tau_{(j)}\}=\{T_\hbar^n\ge\tau_{(j)}\}$, we have 
\begin{multline*}E_{\mathbf{t}}\bigg\{\int_{\tau_{(j)}\wedge T_\hbar^n}^{\tau_{(j+1)}\wedge T_\hbar^n}(\alpha_i(s;\omega))^2\bigg(\frac{\partial f_{\mathcal{S}(\mathbf{t};\tau_{(j)}),\hbar}}{\partial y_i}\bigg)^2ds|\mathcal{B}_{\tilde{\tau}}\bigg\}\\
=c_i^{-2}E_{\mathbf{t}}\bigg\{\int_{\tau_{(j)}\wedge T_\hbar^n}^{\tau_{(j+1)}\wedge T_\hbar^n}(\alpha_1(s;\omega))^2\bigg(\frac{\partial f_{\mathcal{S}(\mathbf{t};\tau_{(j)}),\hbar}}{\partial y_i}\bigg)^2ds|\mathcal{B}_{\tilde{\tau}}\bigg\}
\le c_i^{-2}n\max_{1\le k\le l}\|f_{\mathcal{S}(\mathbf{t};\tau(k)),\hbar}\|^2,\end{multline*}
which is finite. Here $\|f\|:=\max_{1\le i\le N}\sup_{y\in D_\hbar}|\frac{\partial f}{\partial y_i}|$ for any $f\in C^1(D_\hbar)$. We have that the conditional expectation of the stochastic integral in \eqref{BB} is 0. Moreover,  
since $f_{\mathcal{S}(\mathbf{t};\tau_{(j)}),\hbar}(\cdot)$ satisfies the Neumann condition \eqref{cond1}, and the measure $dm_i(s)=0$ off the random set $\{s\,:\,y_i(s)=0\}$, we have that $\int_{\tau_{(j)}\wedge T_\hbar^n}^{\tau_{(j+1)}\wedge T_\hbar^n}\frac{\partial f_{\mathcal{S}(\mathbf{t};\tau_{(j)}),\hbar}}{\partial y_i}dm_i(s)=0$, $P_{\mathbf{t}}$-a.s. for any $i\in\{1,\ldots,N\}$. Therefore, from \eqref{Ito's expansion} and \eqref{BB} we have
\begin{multline}\label{Ito}
E_{\mathbf{t}}\{f_{\mathcal{S}(\mathbf{t};\tau_{(j)}),\hbar}(Y(\tau_{(j+1)}\wedge T_\hbar^n))-f_{\mathcal{S}(\mathbf{t};\tau_{(j)}),\hbar}(Y(\tau_{(j)}\wedge T_\hbar^n))|\mathcal{B}_{\tilde{\tau}}\}\idc{T_\hbar^n\ge\tilde{\tau}}\\
=-E_{\mathbf{t}}\bigg\{\int_{\tau_{(j)}\wedge T_\hbar^n}^{\tau_{(j+1)}\wedge T_\hbar^n}\frac{1}{2}(\alpha_1(s;\omega))^2ds|\mathcal{B}_{\tilde{\tau}}\bigg\}\idc{T_\hbar^n\ge\tilde{\tau}}.
\end{multline}
From Lemma \ref{pde lemma} we know that $0\le f_{\mathcal{S}(\mathbf{t};\tau_{(j)}),\hbar}(y_1,\ldots,y_N)$'s are decreasing in $y_i$ and $f_{\mathcal{S}(\mathbf{t};\tau_{(j)}),\hbar}|_{y_i=h_i}=0$ for each $i\in\{1,\ldots,N\}$, we have that 
\begin{eqnarray*}
&&E_{\mathbf{t}}\bigg\{\int_{\tau_{(j)}\wedge T_\hbar}^{\tau_{(j+1)}\wedge T_\hbar^n}\frac{1}{2}(\alpha_1(s;\omega))^2ds|\mathcal{B}_{\tilde{\tau}}\bigg\}\idc{T_\hbar^n\ge\tilde{\tau}}\\
&=&E_{\mathbf{t}}\{f_{\mathcal{S}(\mathbf{t};\tau_{(j)}),\hbar}(Y(\tau_{(j)}\wedge T_\hbar^n))-f_{\mathcal{S}(\mathbf{t};\tau_{(j)}),\hbar}(Y(\tau_{(j+1)}\wedge T_\hbar^n))|\mathcal{B}_{\tilde{\tau}}\}\idc{T_\hbar^n\ge\tilde{\tau}}\\
&\le& E_{\mathbf{t}}\{f_{\mathcal{S}(\mathbf{t};\tau_{(j)}),\hbar}(Y(\tau_{(j)}\wedge T_\hbar^n))|\mathcal{B}_{\tilde{\tau}}\}\le f_{\mathcal{S}(\mathbf{t};\tau_{(j)}),\hbar}(0,\ldots,0)<\infty.
\end{eqnarray*}
As $n\to\infty$, $t_n$ tends to $\infty$ and $T_\hbar^n$ tends to $T_\hbar$, $P_{\mathbf{t}}$-a.s. Using the bounded convergence theorem we have that 
\begin{align*}\infty>f_{\mathcal{S}(\mathbf{t};\tau_{(l)}),\hbar}(0,\ldots,0)\ge &E_{\mathbf{t}}\bigg\{\int_{\tau_{(l)}\wedge T_\hbar}^{T_\hbar}\frac{1}{2}(\alpha_1(s;\omega))^2ds|\mathcal{B}_{\tilde{\tau}}\bigg\}\idc{T_\hbar\ge\tilde{\tau}}\\
\ge& E_{\mathbf{t}}\bigg\{\idc{T_\hbar=\infty}\int_{\tau_{(l)}}^{\infty}\frac{1}{2}(\alpha_1(s;\omega))^2ds|\mathcal{B}_{\tilde{\tau}}\bigg\}.\end{align*}
We now claim that $P_{\mathbf{t}}(T_\hbar=\infty|\mathcal{B}_{\tilde{\tau}})=0$. This is because, otherwise, from the above inequality we would have
\[\int_{\tau_{(l)}}^\infty(\alpha_1(s;\omega))^2ds=0,\,\,\,P_{\mathbf{t}}\text{-a.s. on }\{T_\hbar=\infty\},\]
which contradicts  Lemma \ref{finite lemma}.
This proves \eqref{finite1}.

Finally, using the bounded convergence theorem and \eqref{finite1}, we can obtain \eqref{KL1} from \eqref{Ito}.

Using similar arguments we can prove \eqref{finite2} and \eqref{KL2}.\qquad\end{proof}

Using Lemma \ref{pde lemma} and Theorem \ref{exp cusum}, we can explicitly compute the $J_{KL}^{(N)}$ index for the multi-chart CUSUM stopping rule $T_\hbar$. To this end, we introduce 
$\mathcal{S}^{(j)}=(\mathcal{S}^{(j)}_1,\ldots,\mathcal{S}^{(j)}_N)$
 with $S_i^{(j)}=\idc{i=j}-\idc{i\not=j}$, for any $i=1,\ldots, N$ and $j=0,1,\ldots, N$.
 We then have

\begin{proposition}\label{NCUSUMEQ}
  For the multi-chart stopping rule $T_\hbar$, we have that
  \begin{eqnarray}\label{worstDelay}
   J_{j}^{(N)}(T_\hbar)=f_{\mathcal{S}^{(j)},\hbar}(0,\ldots,0),\,\,
   j\in\{1,\ldots,N\},
  \end{eqnarray}
  where $J_j^{(N)}$ is defined in \eqref{JjN}.
  \end{proposition}
  \begin{proof}First recall from \eqref{JjN} that, for all $j\in\{1,\ldots,N\}$,
    \begin{eqnarray}J_j^{(N)}(T_\hbar)&=&\mathop{\sup_{\mathbf{t}=(\tau_1,\ldots,\tau_N)\in(\ol{\mathbb{R}}_+)^N}}_{\tau_j=\tilde{\tau}<\infty} ~\textup{essup}~E_{\mathbf{t}}\bigg\{\idc{T_\hbar\ge\tau_j}\int_{\tau_j}^{T_\hbar}\frac{1}{2}(\alpha_1(s;\omega))^2ds|\mathcal{B}_{\tau_j}\bigg\}.\nn\end{eqnarray}
    Without loss of generality, we give the proof for the case $N=2$. The argument can be easily generalized to treat a general $N$. Notice that for $N=2$, we have for $j=1,2$ that,
    \begin{eqnarray} J_j^{(2)}(T_\hbar)=\mathop{\sup_{\mathbf{t}=(\tau_1,\tau_2)\in(\ol{\mathbb{R}}_+)^2}}_{\tau_j=\tilde{\tau}<\infty}\textup{essup}\,E_{\mathbf{t}}\bigg\{\idc{T_\hbar\ge\tau_j}\int_{\tau_j}^{T_\hbar}\frac{1}{2}(\alpha_1(s;\omega))^2ds|\mathcal{B}_{\tau_j}\bigg\}\label{crazyeq}.\end{eqnarray}
We assume without loss of generality that  $\tau_1=\tilde{\tau}$. To prove the claim in the proposition, we first show that, for any fixed $\mathbf{t}=(\tau_1,\tau_2)$ such that $\tau_1\in\mathbb{R}_+$ and $\tau_1=\tilde{\tau}$, we have that
    \begin{eqnarray}E_{\mathbf{t}}\bigg\{\idc{T_\hbar\ge\tau_1}\int_{\tau_1}^{T_\hbar}\frac{1}{2}(\alpha_1(s;\omega))^2ds|\mathcal{B}_{\tau_1}\bigg\}\le f_{(1,-1),\hbar}(0,0),\,\,P_{\mathbf{t}}\text{-a.s.}\label{CCC}\end{eqnarray}
    To show this, we use Theorem \ref{exp cusum} to obtain
    \begin{align*}&E_{\mathbf{t}}\bigg\{\int_{\tau_1}^{T_\hbar}\frac{1}{2}(\alpha_1(s;\omega))^2ds|\mathcal{B}_{\tau_1}\bigg\}\idc{T_\hbar\ge\tau_1}\\
    =&E_{\mathbf{t}}\bigg\{\int_{\tau_1}^{T_\hbar\wedge\tau_2}\frac{1}{2}(\alpha_1(s;\omega))^2ds+\int_{T_\hbar\wedge\tau_2}^{T_\hbar}\frac{1}{2}(\alpha_1(s;\omega))^2ds|\mathcal{B}_{\tau_1}\bigg\}\idc{T_\hbar\ge\tau_1}\\
      =&\idc{T_\hbar\ge\tau_1}\left(E_{(\tau_1,\infty)}\{\idc{T_\hbar\ge\tau_2}[f_{(1,-1),\hbar}(Y(\tau_1))-f_{(1,-1),\hbar}(Y(\tau_2))]|\mathcal{B}_{\tau_1}\}\right.\\
      &+\left.E_{(\tau_1,\infty)}\{\idc{T_\hbar\ge\tau_2}f_{(1,1),\hbar}(Y(\tau_2))|\mathcal{B}_{\tau_1}\}+E_{(\tau_1,\infty)}\{\idc{T_\hbar<\tau_2}f_{(1,-1),\hbar}(Y(\tau_1))|\mathcal{B}_{\tau_1}\}\right).
    \end{align*}
     Using Lemma \ref{pde lemma}, we have that
    \[f_{(1,1),\hbar}(Y(\tau_2))\idc{T_\hbar\ge\tau_2}\le f_{(1,-1),\hbar}(Y(\tau_2))\idc{T_\hbar\ge\tau_2}\le f_{(1,-1),\hbar}(0,0)\idc{T_\hbar\ge\tau_2},~P_{(\tau_1,\infty)}\text{-a.s.}\]
    from which we obtain \eqref{CCC}.
    On the other hand, apply Theorem \ref{exp cusum} to $(\tau_1,\tau_2)=(0,\infty)$, we have
        \begin{align*}E_{0,\infty}\bigg\{\frac{1}{2}\int_{0}^{T_\hbar}(\alpha_1(s;\omega))^2ds\bigg\}=f_{(1,-1),\hbar}(0,0).
\end{align*}
This suggests that (from \eqref{crazyeq}),
\[J_1^{(2)}(T_\hbar)=f_{(1,-1),\hbar}(0,0).\]
Therefore,  \eqref{worstDelay} holds for $j=1$ and and $N=2$. \qquad
\end{proof}

We now establish a upper bound for $\inf_{T\in\mathbb{B}_\gamma}J_{KL}^{(N)}(T)$ using the $N$-CUSUM rule.  First, let us define 
\[H(\gamma)=\{\hbar\in (\mathbb{R}_+)^N\,:\, f_{\mathcal{S}^{(0)},\hbar}(0,\ldots,0)\ge(c_N)^2\gamma\}.\]
The asymptotic analysis in the next section will show that $H(\gamma)\not=\emptyset$. From \eqref{KL2} in Theorem \ref{exp cusum}, we know that for any $\hbar\in H(\gamma)$,
\begin{eqnarray*}E_{\mathbf{t}_\infty}\bigg\{\int_0^{T_\hbar}\frac{1}{2}(c_N^{-1}\alpha_1(s;\omega))^2ds\bigg\}
&=&\frac{1}{(c_N)^2}f_{\mathcal{S}^{(0)},\hbar}(0,\ldots,0)\,\ge\,\gamma.
\end{eqnarray*}
Hence, $T_\hbar\in\mathbb{B}_\gamma$ for any $\hbar\in H(\gamma)$, and we trivially have
\begin{eqnarray}\label{UBJ}
J_{KL}^{(N)}(T_\hbar) & \ge &\inf_{T\in\mathbb{B}_\gamma}J_{KL}^{(N)}(T),\,\,\,\forall\hbar\in H(\gamma).
\end{eqnarray}
From \eqref{EQR}  and Proposition \ref{NCUSUMEQ}, we will look for a vector $\hbar\in H(\gamma)$ for the multi-chart CUSUM stopping rule (18) so that 
\begin{eqnarray}
\label{equalizerexpectation}
f_{\mathcal{S}^{(1)},\hbar}(0,\ldots,0)=f_{\mathcal{S}^{(j)},\hbar}(0,\ldots,0)+o(1),\,\,\,\forall j\in\{1,2,\ldots,N\},
\end{eqnarray}
as $\gamma\to\infty$.

While it is not clear whether or not there is a  $\hbar\in H(\gamma)$ such that \eqref{equalizerexpectation} holds and $f_{\mathcal{S}^{(0)},\hbar}(0,\ldots,0)=(c_N)^2\gamma$, we will give an explicit linear constraint on all coordinates of $\hbar$, namely, \eqref{thres}, under which \eqref{equalizerexpectation} holds asymptotically as $\hbar\to\infty$. The stopping rule chosen in this way gives us an easily computable upper bound for $\inf_{T\in\mathbb{B}_\gamma}J_{KL}^{(N)}(T)$. We will also prove in later sections that this stopping rule is asymptotically optimal as $\gamma\to\infty$.

\section{Asymptotic analysis}\label{Asymptotic analysis}

In this section we will analyze the asymptotic behavior of $f_{\mathcal{S}^{(j)},\hbar}(0,\ldots,0)$, as all $h_i$'s tend to $\infty$.
As a byproduct, we show that the set $H(\gamma)\neq\emptyset$. The key analysis is based on the following integral representation of $f_{\mathcal{S}^{(j)},\hbar}$.
 
 \begin{lemma}\label{integral lemma}
For any $\hbar\in(0,\infty)^N$, and any $y_i\in[0,h_i]$ for all $i\in\{1,\ldots,N\}$,
we have
\begin{eqnarray}
\nn f_{\mathcal{S},\hbar}(y_1,\ldots,y_N) & = & \int_0^{\infty}\prod_{i=1}^NK_{\mathcal{S}_i,\epsilon_i}\bigg(\frac{\epsilon_it}{c_i^2},\frac{y_i}{h_i}\bigg)\,dt,
\end{eqnarray}
where $\epsilon_i:=\frac{1}{h_i}$ and $K_{\mathcal{S}_i,\epsilon_i}(t,z)$ is the  $C^2$ function on $[0,1]$ satisfying the following PDE: for any $z\in [0,1)$,
\begin{subnumcases}{}
\label{PDE_Gi}
\frac{\partial K_{\mathcal{S}_i,\epsilon_i}}{\partial t} =
\epsilon_i \frac{\partial^2K_{\mathcal{S}_i,\epsilon_i}}{\partial z^2}+\mathcal{S}_i \frac{\partial K_{\mathcal{S}_i,\epsilon_i}}{\partial z},\\
\label{PDE_GiCond}
K_{\mathcal{S}_i,\epsilon_i}|_{t=0}=\frac{\partial K_{\mathcal{S}_i,\epsilon_i}}{\partial z}\bigg|_{z=0}=K_{\mathcal{S}_i,\epsilon_i}|_{z=1}=0.
\end{subnumcases}
Moreover, for any $c,\epsilon> 0$, $K_{\pm1,\epsilon}(t,0)$ is a decreasing function, bounded between 0 and 1, and satisfies  
\begin{equation}
\label{G integrals}
\int_0^\infty K_{+1,\epsilon}(c^{-2}\epsilon t,0)dt=c^2\bigg(\frac{1}{\epsilon}-1+\textrm{e}^{-1/\epsilon}\bigg).
\end{equation}
\end{lemma}

To examine the asymptotic behavior of $K_{\mathcal{S}_i,\epsilon_i}(c_i^{-2}\epsilon_it,0)$ as $h_i=\frac{1}{\epsilon_i}\to\infty$ or $\epsilon_i\downarrow0$, 
we transform \eqref{PDE_Gi} to a
heat equation and then obtain an expansion of $K_{\mathcal{S}_i,\epsilon_i}(t,0)$ using the eigenfunctions of the
associated Sturm-Liouville problem. Details of this expansion are given
in the appendix. Using the asymptotic expansion, we obtain an estimate for $f_{\mathcal{S}^{(0)},\hbar}(0,\ldots,0)$ as the $h_i$'s tend to $\infty$.

\begin{proposition}\label{falsealarmasym}
As $h_i=\frac{1}{\epsilon_i}\to\infty$, we have
\begin{eqnarray}\label{fleadingterm}
\bigg|f_{\mathcal{S}^{(0)},\hbar}(0,\ldots,0)-\frac{1}{\sum_{i=1}^Nc_i^{-2}\textup{e}^{-h_i}}\bigg|&\le&\mathcal{O}(1/\epsilon_{\max}),
\end{eqnarray}
where $\epsilon_{\max}=\max_i\{\epsilon_i\}$.
\end{proposition}
Clearly, the estimate in \eqref{fleadingterm} tends to $\infty$ as the $h_i$'s tend to $\infty$, hence we have $H(\gamma)\neq\emptyset$.

Similarly, using Lemma \ref{integral lemma} for $\mathcal{S}=\mathcal{S}^{(j)}$, $j=1,\ldots,N$, we have
\begin{eqnarray}
\nn
f_{\mathcal{S}^{(j)},\hbar}(0,\ldots, 0) & = & \int_0^\infty K_{+1,\epsilon_j}(c_j^{-2}\epsilon_jt,0) \cdot \bigg(\prod_{i\neq j} K_{-1,\epsilon_i}(c_i^{-2}\epsilon_it,0)\bigg)dt\,.
\end{eqnarray}
In the limit for $\epsilon_j\rightarrow 0,\epsilon_i\rightarrow 0$,
the value of the integral is governed by the integral of the function $K_{+1,\epsilon_j}(c_j^{-2}\epsilon_jt,0)$ 
in
the sense that
\begin{equation}\nn
f_{\mathcal{S}^{(j)},\hbar}(0,\ldots, 0) = \int_0^\infty K_{+1,\epsilon_j}(c_j^{-2}\epsilon_jt,0)  dt+o(1)\,=c_j^2(h_j-1)+o(1).
\end{equation} 
More precisely, we have the following result.
\begin{proposition}\label{prop:bound_delay}
As $h_i=\frac{1}{\epsilon_i}\to\infty$, we have
\begin{eqnarray}\label{giapprox}
|f_{\mathcal{S}^{(j)},\hbar}(0,\ldots,0)-c_j^{2}(h_j-1)|\le \mathcal{O}\bigg(\frac{1}{\epsilon_j^2}\sum_{i=1}^N\mathrm{e}^{-1/\epsilon_i}\bigg)\,.
\end{eqnarray}
\end{proposition}

\section{Asymptotic optimality of the $N$-CUSUM rule}

In this section we prove the asymptotic optimality of the $N$-CUSUM rule as $\gamma\to\infty$. This is accomplished by establishing a lower bound for 
$\inf_{T\in\mathbb{B}_\gamma}J_{KL}^{(N)}(T)$, and then showing that the difference between this lower bound and the upper bound  in \eqref{UBJ}  is bounded as $\gamma\to\infty$.

\begin{proposition}
\label{lemmaLB}
For any stopping rule $T\in\mathbb{B}_\gamma$, we have $J_{KL}^{(N)}(T) \ge g(-\nu^\star)$, where $\nu^\star$ is the unique positive solution to $g(\nu^\star)=(c_N)^2\gamma$ and the function $g$ is as defined in equation \eqref{g}. 
\end{proposition}
\begin{proof}
We begin by defining an auxiliary  performance measure
\begin{eqnarray}\label{new index}
J_{KL}^{(N,1)}(T) & = & \mathop{\sup_{\tau<\infty}} ~\textrm{essup}~E_{(\tau,\infty,\ldots,\infty)}\bigg\{\idc{T\ge\tau}\int_{\tau}^{T}\frac{1}{2}(\alpha_1(s;\omega))^2ds|\mathcal{B}_{\tau}\bigg\},
\end{eqnarray}
where $T$ is an arbitrary $\mathbb{B}$-stopping rule.
Notice that, for any $\mathbb{F}$-stopping rule $T\in\mathbb{B}_\gamma$, by construction, we have that
\begin{eqnarray}
\nn J_{KL}^{(N)}(T) \ge J_{KL}^{(N,1)}(T).\end{eqnarray}

Now let $T_\nu^{(1)}$ be the first passage time of $y_1$ to some threshold $\nu>0$
\[T_\nu^{(1)}:=\inf\{t\ge0\,:\,y_1(t)\ge\nu\}.\]
Then trivially, we have $T_\nu:=T\wedge T_\nu^{(1)}\le T$, and that
\begin{eqnarray}
\nn J_{KL}^{(N,1)}(T)\ge J_{KL}^{(N,1)}(T_\nu).
\end{eqnarray} 
To finish the proof, we will show that for any fixed $\epsilon>0$, there exists a $\nu>0$ such that
\begin{eqnarray}
\label{INEQ2}
J_{KL}^{(N,1)}(T_\nu)\ge g(-\nu^\star)-\epsilon.\end{eqnarray}

To that end, we use similar arguments as in Theorem \ref{exp cusum},  for $\tau_1=\tau<\infty=\tau_2,\ldots,\tau_N$, by applying It\^o's lemma to $\{g(-y_1(t))\}_{\tau\wedge T_\nu\le t<T_\nu}$ to obtain that $T_\nu<\infty$, $P_{(\tau,\infty,\ldots,\infty)}$-a.s., and on the event $\{T_\nu>\tau\}$,
\begin{equation}\label{Ito 1d}
E_{(\tau,\infty,\ldots,\infty)}\bigg\{\int_{\tau}^{T_{\nu}}\frac{1}{2}(\alpha_1(s;\omega))^2ds|\mathcal{B}_{\tau} \bigg\} 
= E_{(\tau,\infty,\ldots,\infty)}\{g(-y_1(T_\nu))-g(-y_1(\tau))|\mathcal{B}_{\tau}\}.
\end{equation}
Using Girsanov's theorem at the finite stopping rule $T_\nu\wedge n$ for some $n>\tau$, we have that, 
\begin{eqnarray}
\label{Girsanov1}
\idc{T_\nu\ge \tau}=\idc{T_\nu\ge \tau}E_{\mathbf{t}_\infty}\{\textrm{e}^{u_1(T_\nu\wedge n)-u_1(\tau)}|\mathcal{B}_{\tau}\},
\end{eqnarray}
and 
\begin{multline}
\label{Girsanov2}
\idc{T_\nu\ge \tau}E_{(\tau,\infty,\ldots,\infty)}\{g(-y_1(T_\nu\wedge n)-g(-y_1(\tau))|\mathcal{B}_{\tau}\}\\=\idc{T_\nu\ge \tau}E_{\mathbf{t}_\infty}\{\textrm{e}^{u_1(T_\nu\wedge n)-u_1(\tau)}[g(-y_1(T_\nu\wedge n))-g(-y_1(\tau))]|\mathcal{B}_{\tau}\}.
\end{multline}
Notice that $g(-y)$ is increasing,  $y_1(t\wedge T_\nu)\in[0,\nu]$ and $u_1(T_\nu\wedge n)-u_1(\tau)\le u_1(T_\nu\wedge n)-m_1(T_\nu\wedge n)=y_1(T_\nu\wedge n)\le \nu$, we have that
\begin{eqnarray*}\textrm{e}^{u_1(T_\nu\wedge n)-u_1(\tau)}&\le&\textrm{e}^\nu<\infty,\\
|g(-y_1(T_\nu\wedge n))-g(-y_1(\tau))|&\le&2g(-\nu)<\infty.
\end{eqnarray*}
Using the bounded convergence theorem to \eqref{Girsanov1}-\eqref{Girsanov2} as $n\to\infty$,  we obtain that, on the event $\{T_\nu\ge \tau\}$,
\begin{align*}
1=&E_{\mathbf{t}_\infty}\{\textrm{e}^{u_1(T_\nu)-u_1(\tau)}|\mathcal{B}_{\tau}\},\,\,\,\textrm{and}\\
E_{(\tau,\infty,\ldots,\infty)}\{g(-y_1(T_\nu))-&g(-y_1(\tau))|\mathcal{B}_{\tau}\}\\=&E_{\mathbf{t}_\infty}\{\textrm{e}^{u_1(T_\nu)-u_1(\tau)}[g(-y_1(T_\nu))-g(-y_1(\tau))]|\mathcal{B}_{\tau}\}.
\end{align*}
 In view of \eqref{new index} and \eqref{Ito 1d}, we have that (by substituting $s=\tau$)
\begin{multline}\label{ineq}
  J_{KL}^{(N,1)}(T_\nu)E_{\mathbf{t}_\infty}\{\textrm{e}^{u_1(T_\nu)-u_1(s)}|\mathcal{B}_{s}\}\idc{T_\nu\ge s}\\
  \ge E_{\mathbf{t}_\infty}\{\textrm{e}^{u_1(T_\nu)-u_1(s)}[g(-y_1(T_\nu))-g(-y_1(s))]|\mathcal{B}_{s}\}\idc{T_\nu\ge s}.  \end{multline}
Following the same arguments as in Theorem 2 of \cite{Mous04}, we integrate both sides of the above inequality with respect to $(-dm_1(s))$ for all $s\in[0,T_\nu]$, then take the expectation under $P_{\mathbf{t}_\infty}$. We obtain that 
\begin{multline*}
J_{KL}^{(N,1)}(T_\nu)E_{\mathbf{t}_\infty}\bigg\{\textrm{e}^{u_1(T_\nu)}\int_0^{T_\nu}\textrm{e}^{-u_1(s)}(-dm_1(s))\bigg\}\\
\ge E_{\mathbf{t}_\infty}\bigg\{\textrm{e}^{u_1(T_{\nu})}\int_0^{T_\nu}\textrm{e}^{-u_1(s)}[g(-y_1(T_\nu))-g(-y_1(s))](-dm_1(s))\bigg\}\end{multline*}
Using the fact that $dm_1(s)=0$ off $\{s : y_1(s)=0\}=\{s : u_1(s)=m_1(s)\}$, and $g(0)=0$, we obtain that
\[J_{KL}^{(N,1)}(T_\nu)E_{\mathbf{t}_\infty}\{\textrm{e}^{y_1(T_\nu)}-\textrm{e}^{u_1(T_\nu)}\}\ge E_{\mathbf{t}_\infty}\{[\textrm{e}^{y_1(T_\nu)}-\textrm{e}^{u_1(T_\nu)}]g(-y_1(T_\nu))\}.\]
On the other hand, let $s=0$ in \eqref{ineq} we have that
\[J_{KL}^{(N,1)}(T_\nu)E_{\mathbf{t}_\infty}\{\textrm{e}^{u_1(T_\nu)}\}\ge E_{\mathbf{t}_\infty}\{\textrm{e}^{u_1(T_\nu)}g(-y_1(T_\nu))\}.\]
Combining the above two inequalities, we have
\begin{eqnarray} \label{INEQ3}
J_{KL}^{(N,1)}(T_\nu) {E_{\mathbf{t}_\infty}\{\textrm{e}^{y_1(T_\nu)}\}}\ge {E_{\mathbf{t}_\infty}\{\textrm{e}^{y_1(T_\nu)}g(-y_1(T_\nu))\}}.
\end{eqnarray}

Similarly,  by applying It\^o's lemma to $\{g(y_1(t))\}_{0\le t<T_\nu}$, we  obtain that $P_{\mathbf{t}_\infty}(T_\nu<\infty|\mathcal{B}_s)=1$ for any fixed $s\in\mathbb{R}_+$, and
\begin{eqnarray}
\nn E_{\mathbf{t}_\infty}\bigg\{\int_0^{T_\nu}\frac{1}{2}(\alpha_1(s;\omega))^2ds\bigg\}&=&E_{\mathbf{t}_\infty}\{g(y_1(T_\nu))\}.\end{eqnarray}
On the other hand, using monotone convergence theorem, as $\nu\to\infty$, we have $T_\nu\uparrow T$, and 
\[\lim_{\nu\to\infty}E_{\mathbf{t}_\infty}\bigg\{\int_0^{T_\nu}\frac{1}{2}(\alpha_1(s;\omega))^2ds\bigg\}\,=\,E_{\mathbf{t}_\infty}\bigg\{\int_0^{T}\frac{1}{2}(\alpha_1(s;\omega))^2ds\bigg\}\ge (c_N)^2\gamma.\]
As a result, there exists a $\nu$ large enough, such that 
\[E_{\mathbf{t}_\infty}\{g(y_1(T_\nu))\}\,=\,E_{\mathbf{t}_\infty}\bigg\{\int_0^{T_\nu}\frac{1}{2}(\alpha_1(s;\omega))^2ds\bigg\}\,\ge\,(c_N)^2\gamma-\epsilon\]

Now let us consider the non-negative function $p(y)\,:=\,\textrm{e}^{y}[g(-y)-g(-\nu^\star)]-g(y)+g(\nu^\star)$. We trivially have that $E_{\mathbf{t}_\infty}\{p(y_1(T_\nu))\}\ge0$, which implies that
\begin{align}
 &E_{\mathbf{t}_\infty}\{\textrm{e}^{y_1(T_\nu)}g(-y_1(T_\nu))\}\ge E_{\mathbf{t}_\infty}\{\textrm{e}^{y_1(T_\nu)}\}g(-\nu^\star)+E_{\mathbf{t}_\infty}\{g(y_1(T_\nu))\}-g(\nu^\star)\nn\\
=  &E_{\mathbf{t}_\infty}\{\textrm{e}^{y_1(T_\nu)}\}g(-\nu^\star)+E_{\mathbf{t}_\infty}\{g(y_1(T_\nu))\}-(c_N)^2\gamma
\ge E_{\mathbf{t}_\infty}\{\textrm{e}^{y_1(T_\nu)}\}g(-\nu^\star)-\epsilon\nn\\
\ge& E_{\mathbf{t}_\infty}\{\textrm{e}^{y_1(T_\nu)}\}[g(-\nu^\star)-\epsilon],\nn
\end{align}
since $E_{\mathbf{t}_\infty}\{\textrm{e}^{y_1(T_\nu)}\}\ge1$. The above inequality and \eqref{INEQ3} together imply \eqref{INEQ2}. This completes the proof. \qquad
\end{proof}

In view of \eqref{decomp}, \eqref{UBJ}, Proposition \ref{NCUSUMEQ} and Proposition \ref{lemmaLB}, we have the following inequalities:
\begin{equation}
\label{ULB}
\max\{f_{\mathcal{S}^{(1)},\hbar}(0,\ldots,0),\ldots,f_{\mathcal{S}^{(N)},\hbar}(0,\ldots,0)\}=J_{KL}^{(N)}(T_\hbar) \ge \inf_{T\in\mathbb{B}_\gamma}J_{KL}(T) \ge g(-\nu^\star),
\end{equation}
where $\hbar$ satisfies \eqref{thres} and 
\begin{equation}\label{FRate}
f_{\mathcal{S}^{(0)},\hbar}(0,\ldots,0)=(c_N)^2\gamma,
\end{equation} 
and $\nu^\star$ solves 
\begin{eqnarray}\label{nuFRate}
g(\nu)=(c_N)^2\gamma.\end{eqnarray}
In the sequel we will establish asymptotic optimality of $T_\hbar$ by examining the rate at which the lower and upper bounds in
\eqref{ULB} approach each other as $\gamma \to \infty$ \cite{Hadj06,HadjZhanPoor,HadjSchaPoor}. To this end, 
 we will need to derive the asymptotic behavior of
the quantity  in \eqref{ULB} and \eqref{FRate}. This is done in the following section.

We will distinguish two cases; the first case is the one of signal-strength symmetry and the second one is the
one of signal-strength asymmetry.

\subsection{Signal-strength symmetry}

We capture the signal strength symmetric case by assuming that $|c_i|=1$ for all $i\in\{1,\ldots,N\}$.
In this case, it is natural to consider a simplification of
(\ref{CUSUMmultichart}), namely the one in which $h_1=\ldots=h_N$, since such a choice satisfies (\ref{equalizerexpectation}). For a large fixed $\gamma$, let us denote by $h(\gamma)$ the common threshold determined by \eqref{FRate}. 
The asymptotic analysis in the last section ensures the existence of $h(\gamma)$. The uniqueness of this $h(\gamma)$ follows from Lemma \ref{pde lemma}.
Throughout this subsection, we will slightly abuse our notation to use $h(\gamma)$ to denote the vector $(h(\gamma),\ldots,h(\gamma))$ in this subsection.

First we notice that 
\eqref{FRate}  simplifies to
\begin{eqnarray}
\label{constrsimpl} f_{\mathcal{S}^{(0)},h(\gamma)} (0,\ldots,0)& = & \gamma,
\end{eqnarray}
and \eqref{ULB} simplifies to
\begin{eqnarray}\label{ULBh}
f_{\mathcal{S}^{(1)},h(\gamma)}(0,\ldots,0) \ge \inf_{T\in\mathbb{B}_\gamma}J_{KL}^{(N)}(T) \ge g(-\nu^\star).
\end{eqnarray}

We will demonstrate that the difference between the
upper and lower bounds in \eqref{ULBh}
is bounded by a constant as $\gamma \to \infty$, with $h(\gamma)$ and $\nu^\star$
satisfying \eqref{constrsimpl} and \eqref{nuFRate}, respectively.

\begin{proposition}
\label{asymptotics} 
As $\gamma\to\infty$,
\begin{eqnarray}
\nn f_{\mathcal{S}^{(1)},h(\gamma)}(0,\ldots,0) = \log \gamma +\log
N-1 + o(1).
\end{eqnarray} 
\end{proposition}
\begin{proof} Using \eqref{fleadingterm} and \eqref{giapprox}, the result follows. \qquad\end{proof}

As a result, we have the following result.

\begin{theorem}
\label{main} The difference in the optimal detection delay $\inf_{T\in\mathbb{B}_\gamma}J_{KL}^{(N)}(T)$ and the detection delay of $T_{h(\gamma)}$
of \eqref{CUSUMmultichart} with $h(\gamma)$ satisfying \eqref{constrsimpl} is
bounded above by $\log N$, as $\gamma\to \infty$.
\end{theorem}
\begin{proof}
  From $g(\nu^\star)=\gamma$ it is easily seen that
  \[g(-\nu^\star)=\log\gamma-1+o(1).\]
  The result follows from \eqref{ULBh} and Proposition \ref{asymptotics}.\qquad
\end{proof}

\begin{remark} The asymptotic optimality of the $N$-CUSUM stopping rule in Theorem \ref{main} is of the same strength as the one of Theorem 1 in \cite{HadjZhanPoor}. In other words, the difference in the performance of the unknown optimal stopping rule and the $N$-CUSUM stopping rule is bounded above by a constant as $\gamma \to \infty$. 
\end{remark}

\subsection{Signal-strength asymmetry}

In this section we treat the general case $\max_{i\ge2}|c_i|>1$.
Without loss of generality we can assume that $1=c_1=|c_2|=\ldots=|c_K|$ for some $K<N$ and $\min_{K < i \le N}|c_i|>1$.

Recall that the optimal multi-chart CUSUM thresholds must satisfy \eqref{equalizerexpectation}.  While it seems a formidable task to solve for $\hbar=(h_1,\ldots,h_N)$ explicitly from this constraint and \eqref{FRate}, thanks to the asymptotic analysis in Section \ref{Asymptotic analysis}, we can give a simple explicit linear constraint on $h_i$'s, which will lead to \eqref{equalizerexpectation} asymptotically as $h_i$'s tend to $\infty$.

%

\begin{proposition}
\label{thresholds}
For $\hbar=(h_1,\ldots,h_N)$ such that \begin{eqnarray}
\label{thres}c_1^2 (h_1-1)=c_2^2(h_2-1)=\ldots=c_N^2(h_N-1),
\end{eqnarray}
the condition \eqref{equalizerexpectation} 
 holds asymptotically as $h_1\to\infty$. Moreover,
\begin{eqnarray}
\label{asympthres}
J_{KL}^{(N)}(T_\hbar) & = & h_1-1+o(1),
\end{eqnarray}
as $h_1 \to \infty$.
\end{proposition}
\begin{proof} This is a mere consequence of \eqref{giapprox}. \qquad\end{proof}

\begin{remark}
Proposition \ref{thresholds} is similar to the result obtained in Lemma 2 of \cite{HadjZhanPoor}. More specifically, the result in (\ref{thres}) is similar to the one obtained in Lemma 2 of \cite{HadjZhanPoor}  with $c_i^2=\frac{2}{\mu_i^2}$ for $i=1,2,\dots,N$ and that of \textup{(27)} in \cite{HadjZhanPoor} is the same as that of \eqref{asympthres} for $c_1^2=\frac{2}{\mu_1^2}$ and $c_1=1$. Please notice that the detection delay criterion $J_{KL}^{(N)}$ of \eqref{JKL} incorporates the different signal strengths by including a maximization over all $i\in\{1,\ldots,N\}$.
\end{remark}

It now follows that we have the following result.
\begin{proposition}
\label{AsymptoticDD} For $\hbar$ satisfying \eqref{FRate} and \eqref{thres}, we have
\begin{eqnarray}\label{Jineq}
J_{KL}^{(N)}(T_\hbar) & = & \log\gamma + 2\log|c_N|+\log K-1+ o(1),
\end{eqnarray}
as $\gamma \to \infty$.
\end{proposition}
\begin{proof}The existence and uniqueness of the $\hbar$ determined by  \eqref{FRate}
 and \eqref{thres} can be seem from Lemma \ref{pde lemma} and \eqref{fleadingterm}. We can obtain \eqref{Jineq} from Proposition \ref{thresholds} and \eqref{fleadingterm}.\qquad
\end{proof}

We notice that in the special case in which $K=1$, Proposition \ref{AsymptoticDD} implies that
\begin{eqnarray}
\label{AsymptoticDDK=1}
J_{KL}^{(N)}(T_\hbar) & = & \log \gamma +2\log|c_N| -1 +o(1).
\end{eqnarray}

As a result of Proposition \ref{AsymptoticDD}, we obtain the following two theorems that assert the
asymptotic optimality of $T_\hbar$ of (\ref{CUSUMmultichart}).

\begin{theorem}
\label{mainK=1}
Suppose that $\min_{2 \le i \le N}|c_i|> 1$ and that $\hbar$ satisfies
\eqref{FRate} and \eqref{thres}. Then the optimal detection delay $\inf_{T\in\mathbb{B}_\gamma}J_{KL}^{(N)}(T)$ and the detection delay of $T_\hbar$, $J_{KL}^{(N)}(T_\hbar)$ of \eqref{CUSUMmultichart} converges to $0$
as $\gamma \to \infty$.
\end{theorem}
\begin{proof}The asymptotic lower bound in \eqref{ULB} is $g(-\nu^\star)$ with $g(\nu^\star)=(c_N)^2\gamma$, which implies that
\begin{eqnarray}
  \label{lbasym}
  g(-\nu^\star)=\log\gamma+2\log|c_N|-1+o(1).
  \end{eqnarray}
  From  \eqref{ULB} and \eqref{AsymptoticDDK=1} we obtain
\begin{eqnarray}\nonumber
J_{KL}^{(N)}(T_\hbar)-\inf_{T\in\mathbb{B}_\gamma}J_{KL}^{(N)}(T) & \le &  o(1),
\end{eqnarray}
as $\gamma \to \infty$. \qquad \end{proof}

\begin{theorem}
\label{maingeneralK} Suppose that $\min_{K < i \le N}|c_i|> |c_1|=\ldots=|c_K|=1$ and that $\hbar$ satisfies
\eqref{FRate} and \eqref{thres}. Then the optimal detection delay  $\inf_{T\in\mathbb{B}_\gamma}J_{KL}^{(N)}(T)$ and the detection delay of $T_\hbar$, $J_{KL}^{(N)}(T_\hbar)$ of \eqref{CUSUMmultichart} are bounded above by
$\log K,$
as $\gamma \to \infty$.
\end{theorem}
\begin{proof} The asymptotic lower bound in \eqref{ULB} is $g(-\nu^\star)$. From  \eqref{ULB}, \eqref{lbasym}  and Proposition \ref{AsymptoticDD} we obtain
\begin{eqnarray}\nonumber
J_{KL}^{(N)}(T_\hbar)-\inf_{T\in\mathbb{B}_\gamma}J_{KL}^{(N)}(T) & \le & \log K+ o(1),
\end{eqnarray}
as $\gamma \to \infty$. \qquad\end{proof}

\begin{remark}The results of Theorems \ref{mainK=1}  and \ref{maingeneralK} show the same strength in asymptotic optimality of the $N$-CUSUM rule \eqref{CUSUMmultichart} for $c_1^2=\frac{2}{\mu_1^2}$ and $c_1=1$ as that of the proposed rule in \cite{HadjZhanPoor}. 
\end{remark}

\section{Conclusion}

In this paper we have established asymptotic optimality of the $N$-CUSUM
stopping rule (\ref{CUSUMmultichart}) as the solution to the
stochastic optimization problem of (\ref{eqnproblemKLG}). This
asymptotic optimality is summarized in Theorems \ref{main},
\ref{mainK=1} and \ref{maingeneralK}. The resemblance of these results
to the results of \cite{HadjZhanPoor} are mathematically not
surprising since the Kullback-Leibler information number is included
in the criterion and thus does not appear in the formulas for the
detection delay of $T_\hbar$. Yet, the problem in this paper is
strikingly more general than the one treated in
\cite{HadjZhanPoor}. In particular, the processes that are treated in
this paper are general It\^o processes which can capture signals of
general dependencies. More importantly however, in this set-up the
processes are allowed to be coupled with each other, which provides a
natural way in which to capture most systems since it is the case that
signals received in one sensor can affect what is seen by another. The
results in this work are yet another case 
\cite{Hadj05,HadjHernStam,Hadj06,Hadj07}, which captures the
robustness of the $N$-CUSUM stopping rule in the problem of quickest
detection.

\appendix
\setcounter{secnumdepth}{0}
\section{Proof of Lemma \ref{linear}}
Since $\alpha_t^{(i)}$'s are linear functions of $Z_t^{(j)}$'s, we know that $\alpha_t^{(i)}$'s satisfy the Novikov condition \eqref{novikov} from the discussion on page 234 of \cite{LiptShir1}.

To show that \eqref{Energyi} holds, we introduce a new process $Y_t:=\alpha_t^{(1)}$. An important observation is that $\{Y_t\}_{t\ge0}$ is a Ornstein-Uhlenbeck process under under any fixed $P_{\tau_1,\ldots,\tau_N}$, $\tau_i\in\{0,\infty\}$:
\begin{equation}\label{OU}dY_t=\lambda Y_t\,dt+\sigma\, dW_t,\,\,\,\lambda=\bigg(\sum_{i=1}^N\idc{\tau_i=0}c_i\beta_i\bigg)\in\mathbb{R},\,\,\,\sigma=\sqrt{\sum_{i=1}^N\beta_i^2}\in\mathbb{R}_+, \end{equation}
where $\{W_t\}_{t\ge0}$ is a standard Brownian motion under $\mathbb{P}$:
\begin{eqnarray*}W_t&=&\frac{1}{\sqrt{\sum_{i=1}^N\beta_i^2}}\sum_{i=1}^N\beta_i w_t^{(i)}.
\end{eqnarray*}
To facilitate later calculation, let us denote by $Q^{(\lambda)}$ the law of $Y_t$ in \eqref{OU}, by $E^{Q^{(\lambda)}}$ the expectation under $Q^{(\lambda)}$, and $\mathcal{F}^Y_t=\sigma(Y_s\,;\,s\le t)$. It is then sufficient to obtain \eqref{Energyi} from 
\begin{equation}
\label{OUP}
Q^{(\lambda)}\bigg(\int_0^t Y_s^2ds<\infty\bigg)\,=\,1\,=\,Q^{(\lambda)}\bigg(\int_0^\infty Y_s^2ds=\infty\bigg),\,\,\,\forall t\in \mathbb{R}_+.
\end{equation}
In the sequel we prove that \eqref{OUP} holds. 

Using Girsanov theorem, for any $t \in\mathbb{R}_+$ and $q>\frac{3\lambda^2}{2\sigma^2}$,
\begin{align}\label{Laplace}
&E^{Q^{(\lambda)}}\bigg\{\exp\bigg(-q\int_0^t Y_s^2ds\bigg)\bigg\}\\
=&E^{Q^{(0)}}\bigg\{\exp\bigg(-q\int_0^t Y_s^2ds\bigg)\,\frac{dQ^{(\lambda)}}{dQ^{(0)}}\bigg|_{\mathcal{F}_t^Y}\bigg\}\nn\\
=&E^{Q^{(0)}}\bigg\{\exp\bigg(-\bigg(q+\frac{1}{2}\frac{\lambda^2}{\sigma^2}\bigg)\int_0^t Y_s^2ds+\frac{\lambda}{\sigma^2}\int_0^t Y_sdY_s\bigg)\bigg\}\nonumber\\
=&\textrm{e}^{\frac{\lambda}{\sigma^2}Y_0^2-\lambda t}E^{Q^{(0)}}\bigg\{\exp\bigg(-\bigg(q+\frac{1}{2}\frac{\lambda^2}{\sigma^2}\bigg)\int_0^t Y_s^2ds+\frac{\lambda}{\sigma^2}Y_t^2\bigg)\bigg\}\nonumber\\
=&\textrm{e}^{\frac{\lambda}{\sigma^2}Y_0^2-\lambda t}\frac{1}{\sqrt{\cosh(b t)-\frac{2\lambda}{b}\sinh(b t)}}\exp\bigg(-\frac{Y_0^2}{2\sigma^2}\frac{1-\frac{2\lambda}{b}\coth(b t)}{\coth(b t)-\frac{2\lambda}{b}}\bigg)<\infty,\nn
\end{align}
where $b:=\sqrt{2q\sigma^2+\lambda^2}$. Here we used (2.1) on page 18 of \cite{ABM}. Therefore, the first equality of \eqref{OUP} holds. Moreover, let $t\to\infty$ in \eqref{Laplace} and apply bounded convergence theorem, we have that 
\[E^{Q^{(\lambda)}}\bigg\{\exp\bigg(-q\int_0^\infty Y_s^2ds\bigg)\bigg\}=0\]
Hence, the second equality in \eqref{OUP} also holds.

\section{Proof of Lemma \ref{pde lemma}}
To prove Lemma \ref{pde lemma}, the following result will be useful.
\begin{lemma}\label{Ind RV}
Suppose $X$ and $Y$ are two independent random variables, and \textup{Supp}$(X)=\mathbb{R}_+$. Then $P(X<Y)<1$. 
\end{lemma}
\begin{proof}
Since \textup{Supp}$(X)=\mathbb{R}_+$, for any $c\in\mathbb{R}_+$, we have $P(X<c)<1$. Let us denote by $F_Y(y)=P(Y\le y)$, then
\begin{eqnarray*}
P(X<Y)&=&\int_{0+}^\infty P(X<y)dF_Y(y)\\
&=&\int_{0}^{c-} P(X<y)dF_Y(y)+\int_{c}^\infty P(X<y)dF_Y(y)\\
&\le&P(X<c)P(Y<c)+P(Y\ge c)<1.
\end{eqnarray*} 
This completes the proof.\qquad
\end{proof}

We are ready to give the proof of Lemma \ref{pde lemma}.
  Without loss of generality, we prove the lemma for $N=2$. 
    
    The main tool we will use in this proof is Feynman-Kac representation of the solution to an elliptic PDE. In particular, fix a filtered probability space $(\Omega,(\mathcal{F}_t)_{t\ge0}, P)$ with two independent standard Brownian motions $\{B_t^{(1)}\}_{t\ge 0}$ and $\{B_t^{(2)}\}_{t\ge0}$.
    For any $\mathcal{S}=(\mathcal{S}_1,\mathcal{S}_2)\in\{\pm 1\}^2$, $c_2\neq0$, consider two independent reflected Brownian motions defined as
    \begin{eqnarray*}
      X_t^{\mathcal{S}_1}&:=&X_0+\mathcal{S}_1t+\sqrt{2}B_t^{(1)}-\inf_{0\le s\le t}\{(X_0+\mathcal{S}_1s+\sqrt{2}B_s^{(1)})\wedge 0\},\\
      Y_t^{\mathcal{S}_2}&=&Y_0+\mathcal{S}_2 c_2^{-2}t+\sqrt{2}|c_2|^{-1} B_t^{(2)}-\inf_{0\le s\le t}\{(Y_0+\mathcal{S}_2c_2^{-2} s+\sqrt{2}|c_2|^{-1} B_s^{(2)})\wedge0\}.
    \end{eqnarray*}
We denote their first passage times by
\begin{eqnarray*}
T_{a}^{X^{\mathcal{S}_1}}&=&\inf\{t\ge0\,:\,X_t^{\mathcal{S}_1}\ge a\},\,\,\,\forall a\in\mathbb{R}\\
T_b^{Y^{\mathcal{S}_2}}&=&\inf\{t\ge0\,:\,Y_t^{\mathcal{S}_2}\ge b\},\,\,\,\forall b\in\mathbb{R}.
\end{eqnarray*}
For convenience, we will also denote by
\[P^{x,\cdot}(\,\cdot\,)=P(\,\cdot\,|X_0=x),\,P^{\cdot,y}(\,\cdot\,)=P(\,\cdot\,|Y_0=x),\,P^{x,y}(\,\cdot\,)=P(\,\cdot\,|X_0=x, Y_0=y).\]
Then by Feynman-Kac theorem, for any $(x,y)\in[0,h_1]\times[0,h_2]$,
\[0\le E^{x,y}\{T^{X^{\mathcal{S}_1}}_{h_1}\wedge T^{Y^{\mathcal{S}_2}}_{h_2}\}=f_{\mathcal{S},\hbar}(x,y),\]
where $f_{\mathcal{S},\hbar}$  solves \eqref{pde1}-\eqref{cond1}. So it suffices to prove that, for $(x',y')\in[0,x]\times[0,y]$, $0<h_1'<h_1$, $0<h_2'<h_2$, and $\mathcal{S}_i'\le\mathcal{S}_i$, $i=1,2$,  
\begin{eqnarray}
\label{monotone1}E^{x,y}\{T^{X^{\mathcal{S}_1}}_{h_1}\wedge T^{Y^{\mathcal{S}_2}}_{h_2}\}&\le&E^{x',y'}\{T^{X^{\mathcal{S}_1}}_{h_1}\wedge T^{Y^{\mathcal{S}_2}}_{h_2}\}\wedge E^{x,y}\{T^{X^{\mathcal{S}_1'}}_{h_1}\wedge T^{Y^{\mathcal{S}_2'}}_{h_2}\},\\
\label{monotone2}E^{0,0}\{T^{X^{\mathcal{S}_1}}_{h_1}\wedge T^{Y^{\mathcal{S}_2}}_{h_2}\}&>&E^{0,0}\{T^{X^{\mathcal{S}_1}}_{h_1'}\wedge T^{Y^{\mathcal{S}_2}}_{h_2}\}\vee E^{0,0}\{T^{X^{\mathcal{S}_1}}_{h_1}\wedge T^{Y^{\mathcal{S}_2}}_{h_2'}\}.
\end{eqnarray}
In the sequel we will prove \eqref{monotone1} and \eqref{monotone2} hold. 

Using continuity of the sample path and Markov shifting operator, we have
\begin{eqnarray*}T_{h_1}^{X^{\mathcal{S}_1}}&=&T_{x}^{X^{\mathcal{S}_1}}+T_{h_1}^{X^{\mathcal{S}_1}}\circ\theta(T_{x}^{X^{\mathcal{S}_1}})\,\ge\, T_{h_1}^{X^{\mathcal{S}_1}}\circ\theta(T_{x}^{X^{\mathcal{S}_1}}),\,\,\,P^{x',\cdot}\text{-a.s.}\\
T_{h_2}^{Y^{\mathcal{S}_2}}&=&T_{y}^{Y^{\mathcal{S}_2}}+T_{h_2}^{Y^{\mathcal{S}_2}}\circ\theta(T_{y}^{Y^{\mathcal{S}_2}})\,\ge\, T_{h_1}^{Y^{\mathcal{S}_2}}\circ\theta(T_{y}^{Y^{\mathcal{S}_2}}),\,\,\,P^{\cdot,y'}\text{-a.s.}\end{eqnarray*}
Hence we have the first inequality in \eqref{monotone1}. Similarly, for any $y''\in[0,h_2)$,
\[T^{X^{\mathcal{S}_1}}_{h_1}\wedge T^{Y^{\mathcal{S}_2}}_{h_2}>T^{X^{\mathcal{S}_1}}_{h_1'}\wedge T^{Y^{\mathcal{S}_2}}_{h_2}=0,\,\,\,P^{h_1',y''}\text{-a.s.}\]
Thus, 
\[E^{h_1',y''}\{T^{X^{\mathcal{S}_1}}_{h_1}\wedge T^{Y^{\mathcal{S}_2}}_{h_2}\}>0.\]
Using Lemma \ref{Ind RV}  we know that $P^{0,0}(T^{X^{\mathcal{S}_1}}_{h_1'}<T^{Y^{\mathcal{S}_2}}_{h_2})>0$. It follows that
\begin{align*}
&E^{0,0}\{T^{X^{\mathcal{S}_1}}_{h_1}\wedge T^{Y^{\mathcal{S}_2}}_{h_2}\}\\
=&E^{0,0}\{T^{X^{\mathcal{S}_1}}_{h_1'}\wedge T^{Y^{\mathcal{S}_2}}_{h_2}\}+\int_0^{h_2} P^{0,0}(T^{X^{\mathcal{S}_1}}_{h_1'}<T^{Y^{\mathcal{S}_2}}_{h_2}, Y_{T^{X^{\mathcal{S}_1}}_{h_1'}}^{\mathcal{S}_2}\in dy'') \,E^{h_1',y''}\{T^{X^{\mathcal{S}_1}}_{h_1}\wedge T^{Y^{\mathcal{S}_2}}_{h_2}\}\\
>&E^{0,0}\{T^{X^{\mathcal{S}_1}}_{h_1'}\wedge T^{Y^{\mathcal{S}_2}}_{h_2}\}.
\end{align*}
This proves \eqref{monotone2}.

Finally, using  Lemma 3 of \cite{Hadj06}, we have,
\[X_t^{\mathcal{S}_1'}\le X_t^{\mathcal{S}_1}, \,\,\,Y_t^{\mathcal{S}_2'}\le Y_t^{\mathcal{S}_2},\,\,\forall t\ge 0,\,\,~P\text{-a.s.}\]
Hence,
\begin{eqnarray*}T_{h_1}^{X^{\mathcal{S}_1}}\le T_{h_1}^{X^{\mathcal{S}_1'}},\,\,
T_{h_2}^{Y^{\mathcal{S}_2}}\le T_{h_2}^{Y^{\mathcal{S}_2'}},\,\,\,P\text{-a.s.}
\end{eqnarray*}
which implies the second inequality in \eqref{monotone1}. The finiteness of $f_{\mathcal{S},\hbar}(0,0)$ follows from that of $E^{0,\cdot}\{T_{h_1}^{X^{\mathcal{S}_1}}\}$. The latter can be easily shown to be  equal to $g(-\mathcal{S}_1h_1)<\infty$ (see for example, \cite{Mous04}).

\section{Proof of Lemma \ref{finite lemma}}
From \eqref{Energyi} we have that, for any fixed ${\tau}<\infty$ and $\mathbf{t}=(\tau_1,\ldots,\tau_N)\in\{0,\infty\}^N$, $\int_{{\tau}}^\infty\frac{1}{2}(\alpha_1(s;\omega))^2ds=\infty$, $$P_{\mathbf{t}}\bigg(\int_\tau^\infty\frac{1}{2}(\alpha_1(s;\omega))^2ds=\infty\bigg)=1.$$ Now without loss of generality, 
assume that $0\le\tau_1,\ldots,\tau_r<\infty$ and $\tau_{r+1}=\ldots=\tau_N=\infty$, we have from \eqref{RNDeri} that 
\[\frac{dP_{\tau_1,\ldots,\tau_r,\tau_{r+1},\ldots\tau_N}}{dP_{0,\ldots,0,\infty,\ldots,\infty}}\bigg|_{\mathcal{B}_t}=\exp\bigg(-\sum_{i=1}^ru_i(t\wedge\tau_i)\bigg)>0,\,\,\,\forall t\in \mathbb{R}_+.\]
Hence, $P_{\mathbf{t}}(\int_{{t}}^\infty\frac{1}{2}(\alpha_1(s;\omega))^2ds=\infty)=1$ for any $\mathbf{t}=(\tau_1,\ldots,\tau_N)\in(\ol{\mathbb{R}}_+)^N$. Moreover, for any fixed  $\tau'<\infty$, 
\[1=P_{\mathbf{t}}\bigg(\int_{{t}}^\infty\frac{1}{2}(\alpha_1(s;\omega))^2ds=\infty\bigg)=E_{\mathbf{t}}\bigg\{P_{\mathbf{t}}\bigg(\int_{{t}}^\infty\frac{1}{2}(\alpha_1(s;\omega))^2ds=\infty|\mathcal{B}_{{\tau'}}\bigg)\bigg\},\]
there must be that $P_{\mathbf{t}}(\int_{t}^\infty\frac{1}{2}(\alpha_1(s;\omega))^2ds=\infty|\mathcal{B}_{{\tau'}})=1$.

\section{Proof of Lemma \ref{integral lemma}}
Without loss of generality,  we prove the lemma for $N=2$.

Using notations in the proof of Lemma \ref{pde lemma}, we have that 
\begin{eqnarray*}
f_{\mathcal{S},\hbar}(x,y)\,=\,E^{x,y}\{T^{X^{\mathcal{S}_1}}_{h_1}\wedge T^{Y^{\mathcal{S}_2}}_{h_2}\}&=&\int_0^\infty P^{x,y}(T^{X^{\mathcal{S}_1}}_{h_1}>t, T^{Y^{\mathcal{S}_2}}_{h_2}>t)\,dt\\
&=&\int_0^\infty P^{x,\cdot}(T^{X^{\mathcal{S}_1}}_{h_1}>t)\,\cdot P^{\cdot,y}( T^{Y^{\mathcal{S}_2}}_{h_2}>t)\,dt.
\end{eqnarray*}
Hence, it suffices to show that
\begin{eqnarray*}
P^{x,\cdot}(T^{X^{\mathcal{S}_1}}_{h_1}>t)&=&K_{\mathcal{S}_1,\frac{1}{h_1}}\bigg(\frac{c_1^{-2}t}{h_1},\frac{x}{h_1}\bigg),\\
P^{\cdot,y}(T^{Y^{\mathcal{S}_2}}_{h_2}>t)&=&K_{\mathcal{S}_2,\frac{1}{h_2}}\bigg(\frac{c_2^{-2}t}{h_2},\frac{y}{h_2}\bigg).
\end{eqnarray*}
However, it is easily seen that both sides of the above equations satisfy the backward Kolmogorov equation \eqref{PDE_Gi} and boundary condition \eqref{PDE_GiCond}. The assertion now follows from uniqueness theorem for partial differential equations. 
 
 Finally, 
 \[\int_0^\infty K_{+1,\frac{1}{h_1}}\bigg(\frac{c_1^{-2}t}{h_1},0\bigg)dt=\int_0^\infty P^{0,\cdot}(T_{h_1}^{X^{+1}}>t)dt=E^{0,\cdot}\{T_{h_1}^{X^{+1}}\}=g(-h_1).\]
 Hence, \eqref{G integrals} holds for $i=1$. The general case can be similarly proven.

\section{Proofs of Propositions \ref{falsealarmasym} and \ref{prop:bound_delay}}
 We now prove the asymptotic results in Propositions \ref{falsealarmasym} and \ref{prop:bound_delay}. To this end,
we transform the function
$K_{\pm,\epsilon}(t,z)$ defined in \eqref{PDE_Gi}-\eqref{PDE_GiCond} 

to the solution of a
heat equation.  Using Sturm-Liouville theory, we are able to express $K_{\pm1,\epsilon}(t,z)$ as series. From these series, all necessary
asymptotic expansions can be derived directly.
\begin{theorem}
  Let $0<\epsilon<1/2$ and let $K_{-1,\epsilon}$ be the solution to
 \eqref{PDE_Gi}-\eqref{PDE_GiCond} for $\mathcal{S}_i=-1$ and $\epsilon_i=\epsilon>0$. Then, for $t>0$,
\begin{eqnarray}
K_{-1,\epsilon}(t,0) = K_{-1,\epsilon}^{0}(t)+H_\epsilon(t),\,\, H_\epsilon(t):=\sum_{n=1}^{\infty}K_{-1,\epsilon}^n(t),\label{Kminus}
\end{eqnarray}
and the functions $K_{-1,\epsilon}^0, K_{-1,\epsilon}^1,\ldots$, are explicitly given by
\begin{eqnarray}
K_{-1,\epsilon}^{0}(t) &=& A^0_\epsilon\,{\mathrm{e}}^{(\omega^2\epsilon-\frac{1}{4\epsilon}) t}, \label{K0}\\
A^0_\epsilon &=& \frac{1}{\frac{1-{\mathrm{e}}^{-4\omega}}{2\omega}-2{\mathrm{e}}^{-2\omega}}\left(1-{\mathrm{e}}^{-2\omega}\right)^2\frac{1}{\omega+\frac{1}{2\epsilon}}\,
{\mathrm{e}}^{\omega-1/(2\epsilon)},\label{A0}\\
\nn K_{-1,\epsilon}^n(t) &=& A^n_\epsilon\,{\mathrm{e}}^{-1/(2\epsilon)}\,\frac{\sin\omega_n}{\omega_n}\,{\mathrm{e}}^{-(\omega_n^2\epsilon+\frac1{4\epsilon})\,t}\,,~n=1,2,\ldots \\
\nn A^n_\epsilon &=& \frac{8\epsilon^2\omega_n^2}{4\epsilon^2\omega_n^2+1-2\epsilon}\,,
\end{eqnarray}
 where $\omega$ and $\omega_n$'s
are respectively the positive solutions to the transcendental equations
\begin{eqnarray}\label{omegaeq}
\tanh\omega&=&2\epsilon\omega,\\
\tan\omega_n& =& 2\epsilon\omega_n.\label{omeganeq}
\end{eqnarray}
\label{ts_theorem_minus}
\end{theorem}

\begin{proof}
We begin by transforming
(\ref{PDE_Gi}) to the heat equation: direct calculation shows
that the solution $K_{-1,\epsilon}$ can be written as
\begin{equation}
K_{-1,\epsilon}(t,z) = {\mathrm{e}}^{\frac{1}{2\epsilon}z-\frac{1}{4\epsilon}t}u(t,z),\nn
\end{equation}
where $u$ is the solution of
\begin{subnumcases}{}
\frac{\partial u}{\partial t}=\epsilon\frac{\partial^2 u}{\partial z^2},\label{upde1}\\
u(0,z) = {\mathrm{e}}^{-\frac{1}{2\epsilon}z}, \,\,\left.\bigg(\frac{1}{2\epsilon}u+\frac{\partial u}{\partial z}\bigg)\right|_{z=0}=u|_{z=1}=0.\label{upde2}
\end{subnumcases}
The problem \eqref{upde1}-\eqref{upde2} can be solved using Sturm-Liouville theory \cite{WardCheney,Reddy}. More specifically, 
consider the ordinary differential equation $-\epsilon \frac{d^2\hat{u}}{d z^2} = \lambda \hat{u}$, supplemented with the separated boundary conditions
$\hat{u}(1)=0$ and $\hat{u}(0)/(2\epsilon)+\hat{u}'(0)=0$. Clearly, the problem is a regular Sturm-Liouville problem,
such that the corresponding normalized  eigenfunctions form an orthonormal basis
of the Hilbert space ${\mathcal{L}}^2([0,1])$ (see for example, Theorem 12 on page 204 of \cite{Reddy}). Hence, 
$u$ defined in \eqref{upde1}-\eqref{upde2} can be written as
\begin{equation} \label{usum_formula}
u(t,z) = \sum_{n=0}^{\infty}a_n{\mathrm{e}}^{-\lambda_n\epsilon t}\phi_n(z),
\end{equation}
where the $\phi_n$'s are the orthonormal eigenfunctions of the operator $L:=d^2/dx^2$ with the
specified boundary conditions on ${\mathcal{L}}^2([0,1])$ and the $a_n$'s are given as
projections of the initial conditions via
\begin{equation} \label{an_formula}
a_n  = \int_0^1\phi_n(z){\mathrm{e}}^{-z/(2\epsilon)}dz.
\end{equation}
It will be shown below  that $L$ has exactly one positive eigenvalue
$\omega^2$ given by the transcendental equation \eqref{omegaeq} and negative eigenvalues $-\omega_n^2$ given by
the corresponding equations \eqref{omeganeq}. After computing the eigenfunctions
explicitly, the $a_n$'s can be obtained directly. 

We now proceed to compute the positive eigenvalue. To this end, we solve $d^2\phi_0/dz^2 = \omega^2 \phi_0$ with the Dirichlet boundary condition
at $z=1$ and the mixed boundary condition $\phi_0/(2\epsilon)+(d\phi_0/dz) = 0$ at
$z=0$. Clearly, we can write the solution of $\phi_0$ as
\begin{equation} 
\phi_0(z) = A{\mathrm{e}}^{\omega z} + B{\mathrm{e}}^{-\omega z},\nn
\end{equation}
and then use the boundary conditions to find the constants $A, B$ and $\omega$. This yields
the relations
\begin{equation}\label{ABomega}
A = -B{\mathrm{e}}^{-2\omega} ,\qquad \tanh\omega=2\epsilon\omega.\end{equation}
It can be  easily seen that, for $\epsilon<1/2$, we have 
\[\frac{d}{d\omega}\left(\tanh\omega-2\epsilon\omega\right)=\frac{1}{\cosh^2\omega}-2\epsilon,\]
which monotonically decreases from $1-2\epsilon>0$ to $-2\epsilon$ as $\omega$ increases from $0$ to $\infty$. Hence,
the transcendental equation in \eqref{ABomega}
has a unique positive solution, 
which is also the unique positive eigenvalue of $L$. 
The corresponding eigenfunction $\phi$ is then given by
\begin{equation} \label{phi_0_formula}
\phi_0(z) = A\left({\mathrm{e}}^{-\omega z}-{\mathrm{e}}^{-2\omega}{\mathrm{e}}^{\omega z}\right),
\end{equation}
and the constant $A$ is fixed by the normalization requirement
\begin{displaymath}
\int_0^1 \phi_0(z)^2\, dz = 1
\end{displaymath}
yielding after straightforward algebra
\begin{equation}
A = \frac{1}{\sqrt{\frac{1-{\mathrm{e}}^{-4\omega}}{2\omega}-2{\mathrm{e}}^{-2\omega}}},\label{Acoef}
\end{equation}
Using \eqref{usum_formula}, \eqref{an_formula}, \eqref{phi_0_formula} and \eqref{Acoef}, the formulas for $K^{0}_{-1,\epsilon}$ and $A^0_{\epsilon}$ in \eqref{K0} and \eqref{A0} follow directly.
The calculations for the eigenfunctions corresponding to the negative eigenvalues are similar, and we omit it.\qquad
\end{proof}

Using a similar argument as in the proof of Theorem
\ref{ts_theorem_minus}, we can obtain the following result.

\begin{theorem}\label{ts_theorem_plus}
 Let $0<\epsilon<1/2$ and let $K_{+1,\epsilon}$ be the solution to
  \eqref{PDE_Gi}-\eqref{PDE_GiCond}. Then, for $t>0$,
\begin{equation}\label{Kplus1}
K_{+1,\epsilon}(t,0) = {\mathrm{e}}^{-\frac{t}{4\epsilon}}\,\sum_{n=1}^{\infty}K_{+1,\epsilon}^n(t),
\end{equation}
and the functions $K_{+1,\epsilon}^n$ are explicitly given by
\begin{displaymath}
K_{+1,\epsilon}^n(t) = {\mathrm{e}}^{1/(2\epsilon)}\,\frac{\sin\omega_n'}{\omega_n'}\frac{8\epsilon^2(\omega_n')^2}{4\epsilon^2(\omega_n')^2+1+2\epsilon}\,{\mathrm{e}}^{-(\omega_n')^2\epsilon\,t}\,,
\end{displaymath}
 where the $\omega_n'$ are the positive
solutions of the transcendental equation
\begin{displaymath}
\tan\omega_n' = -2\epsilon\omega_n'.
\end{displaymath}
\end{theorem}

Using Theorem \ref{ts_theorem_minus}, it can be shown that the leading term of the asymptotic expansion of the function $K_{-1,\epsilon}$ for small $\epsilon$ is given by  $K_{-1,\epsilon}^0$. In fact, we have:

\begin{corollary}\label{cor:Hbound}
Under the assumptions of Theorem \ref{ts_theorem_minus} we have
\begin{eqnarray}
|H_\epsilon(t)|\le{2}\mathrm{e}^{-1/(2\epsilon)}\sum_{n=1}^\infty\frac{1}{n\pi}\mathrm{e}^{-\epsilon n^2\pi^2t},\,\,\,
\int_0^{\infty}|H_\epsilon(t)|dt\leq \frac{2}{\epsilon\pi^3}{\mathrm{e}}^{-1/(2\epsilon)}\zeta(3),\nn
\end{eqnarray}
where $\zeta(3)$ is Ap\'ery's constant. 
\end{corollary}
\begin{proof} We verify this inequality by direct calculation. Using the explicit
expansion for $H_\epsilon(t)$ in \eqref{Kminus}, we find that
\begin{eqnarray*}
|H_\epsilon(t)|\le\mathrm{e}^{-1/(2\epsilon)} \sum_{n=1}^\infty A_\epsilon^n\frac{1}{\omega_n}\mathrm{e}^{-(\frac{1}{4\epsilon}+\epsilon\omega_n^2)t}\le {2}\mathrm{e}^{-1/(2\epsilon)}\sum_{n=1}^\infty\frac{1}{n\pi}\mathrm{e}^{-\epsilon n^2\pi^2t},
\end{eqnarray*}
where we have made use of the fact that $\omega_n\in[n\pi,n\pi+\pi/2]$. It follows that
\begin{eqnarray*}
\int_0^{\infty}|H_\epsilon(t)|\,dt &\leq& 2 {\mathrm{e}}^{-1/(2\epsilon)}\sum_{n=1}^{\infty} 
\frac{1}{n\pi}\int_0^{\infty}{\mathrm{e}}^{-\epsilon n^2\pi^2t}\,dt =\frac{2{\mathrm{e}}^{-1/(2\epsilon)}}{\epsilon\pi^3}\sum_{n=1}^{\infty} \frac{1}{n^3}=\frac{2{\mathrm{e}}^{-1/(2\epsilon)}}{\epsilon\pi^3}\zeta(3).
\end{eqnarray*}
\qquad
\end{proof}

Similarly, the tail of the integral of $K_{+1,\epsilon}(t,0)$ is exponentially small as $\epsilon\to0+$.
\begin{corollary}\label{Kbound}
Under assumption of Theorem \ref{ts_theorem_plus}, we have
\begin{equation}
\int_{6}^\infty K_{+1,\epsilon}(t,0)dt\le \frac{14}{\epsilon\pi^3}{\mathrm{e}}^{-1/\epsilon}\zeta(3).\nn
\end{equation}
\end{corollary}
\begin{proof}
Using the explicit expansion for $K_{+1,\epsilon}(t,0)$ in \eqref{Kplus1}, we find that
\begin{eqnarray*}
\int_6^{\infty}K_{+1,\epsilon}(t,0)\,dt &\leq& \sum_{n=1}^{\infty} {\mathrm{e}}^{1/(2\epsilon)}
\,\left|\frac{\sin\omega_n'}{\omega_n'}\right|\frac{8\epsilon^2(\omega_n')^2}{4\epsilon^2(\omega_n')^2+1+2\epsilon}\int_6^{\infty}{\mathrm{e}}^{-(\frac1{4\epsilon}+(\omega_n')^2\epsilon)t}\,dt \\
&\leq& \sum_{n=1}^{\infty} 2{\mathrm{e}}^{1/(2\epsilon)}\frac{1}{\omega_n'}\frac{4\epsilon}{1+4\epsilon^2(\omega_n')^2} \mathrm{e}^{-(\frac3{2\epsilon}+6(\omega_n')^2\epsilon)}\\
&<& \sum_{n=1}^{\infty} \frac{2}{\epsilon}{\mathrm{e}}^{-1/\epsilon}\frac{1}{(\omega_n')^3}
\leq \frac{2}{\epsilon}{\mathrm{e}}^{-1/\epsilon}\frac{1}{\pi^3}\sum_{n=1}^{\infty} \frac{8}{(2n-1)^3}= \frac{14}{\epsilon}{\mathrm{e}}^{-1/\epsilon}\frac{1}{\pi^3}\zeta(3),
\end{eqnarray*}
where we have made use of the fact that $\omega_n'\in[n\pi-\pi/2,n\pi]$.\qquad
\end{proof}

\begin{lemma}\label{restimate}
Let $\omega$ be the unique positive solution to $\tanh \omega=2\epsilon \omega$ for $0<\omega<\frac{1}{2}$. Then there exits a constant $C>0$ such that 
\begin{eqnarray}
\bigg|\omega-\frac{1}{2\epsilon}+\frac{1}{\epsilon}\mathrm{e}^{-1/\epsilon}\bigg|\le\frac{C}{\epsilon^2}\mathrm{e}^{-2/\epsilon}.\nn
\end{eqnarray}
\end{lemma}
\begin{proof}
We begin by rewriting the equation that $\omega$ satisfies as
\[f(\omega)=0,\,\,\,\text{where}\,\,\,f(x):=x-\frac{1}{2\epsilon}+\bigg(x+\frac{1}{2\epsilon}\bigg)\mathrm{e}^{-2\epsilon}.\]
We then employ a Newton-Raphson iteration with an initial value $x_0=\frac{1}{2\epsilon}$, and  define $\{x_n\}$ recursively using
\[x_{n+1}=x_n-\frac{f(x_n)}{f'(x_n)},\,\,\,n=0,1,\ldots.\]
The result now follows from the fact that $|f''(x)/f'(x)|$ is uniformly bounded on $\mathbb{R}_+$.\qquad
\end{proof}

\subsection{Proof of Proposition \ref{falsealarmasym}}
We now use Theorem \ref{ts_theorem_minus} to prove Proposition \ref{falsealarmasym}.
To this end, let us introduce 
\begin{eqnarray}
\tilde{f}_{S^{(0)},\hbar}:=\int_0^\infty \prod_{i=1}^NK_{-1,\epsilon_i}^0(c_i^{-2}\epsilon_it)\,dt.
\end{eqnarray}
Using Corollary \ref{cor:Hbound}, it can be shown that the leading term in the asymptotic expansion of $f_{\mathcal{S}^{(0)},\hbar}(0,\ldots,0)$ is  $\tilde{f}_{S^{(0)},\hbar}$, which  can be computed using explicit formulas in Theorem \ref{ts_theorem_minus}. 

Without lose of generality, we prove the results for the case $N=2$. 

Using the fact that $0\le K_{-1,\epsilon}(t,0)\le 1$ for all $t, \epsilon>0$, and that $\int|a(t)b(t)|dt\le\int |a(t)|dt\cdot\int |b(t)|dt$ for any integrable functions $a(t)$ and $b(t)$, we  have 
\begin{align}\label{estimate1}
&|f_{\mathcal{S}^{(0)},\hbar}(0,0)-\tilde{f}_{\mathcal{S}^{(0)},\hbar}|\\
\leq& \int_0^{\infty} |K_{-1,\epsilon_1}(c_1^{-2}\epsilon_1t,0)\cdot H_{\epsilon_2}(c_2^{-2}\epsilon_2t)|dt + \int_0^{\infty} |K_{-1,\epsilon_2}(c_2^{-2}\epsilon_2t,0)\cdot H_{\epsilon_1}(c_1^{-2}\epsilon_1t)|dt\nn \\ 
&+\int_0^{\infty} |H_{\epsilon_1}(c_1^{-2}\epsilon_1t)\cdot H_{\epsilon_2}(c_2^{-2}\epsilon_2t)|dt\nonumber\\
\le&\int_0^{\infty} |H_{\epsilon_2}(c_2^{-2}\epsilon_2t)|dt + \int_0^{\infty} |H_{\epsilon_1}(c_1^{-2}\epsilon_1t)|dt +\int_0^{\infty} |H_{\epsilon_1}(c_1^{-2}\epsilon_1t)\cdot H_{\epsilon_2}(c_2^{-2}\epsilon_2t)|dt\nonumber\\
\le&\int_0^{\infty} |H_{\epsilon_2}(c_2^{-2}\epsilon_2t)|dt + \int_0^{\infty} |H_{\epsilon_1}(c_1^{-2}\epsilon_1t)|dt +\int_0^{\infty} |H_{\epsilon_1}(c_1^{-2}\epsilon_1t)|dt\cdot\int_0^\infty |H_{\epsilon_2}(c_2^{-2}\epsilon_2t)|dt\nonumber\\
\le&\frac{2c_2^2}{\epsilon_2^2\pi^3}\zeta(3)\textup{e}^{-\frac{1}{2\epsilon_2}}+\frac{2c_1^2}{\epsilon_1^2\pi^3}\zeta(3)\textup{e}^{-\frac{1}{2\epsilon_1}}+\frac{4(c_1c_2)^2}{(\epsilon_1\epsilon_2)^2\pi^6}\zeta^2(3)\textup{e}^{-\frac{1}{2\epsilon_1}-\frac{1}{2\epsilon_2}}=\mathcal{O}\bigg(\frac{1}{\epsilon_{\max}^4}\textup{e}^{-\frac{1}{\epsilon_{\max}}}\bigg),\nn
\end{align}
where,  we used Corollary \ref{cor:Hbound} in the last inequality. %

On the other hand, let us denote by $\omega^{(i)}$ the unique positive solution to transcendental equation $\tanh\omega^{(i)}=2\epsilon_i\omega^{(i)}$ for a fixed $\epsilon_i\in(0,\frac{1}{2})$. 
Using Newton-Raphson iteration and Lemma \ref{restimate}, we can write
\begin{equation}\label{estimateomega}
\omega^{(i)} = \frac{1}{2\epsilon_i}-\frac{1}{\epsilon_i}{\mathrm{e}}^{-1/\epsilon_i}+r_i, \,\,\,\text{where}\,\,\, \,|r_i|\leq \frac{C_i}{\epsilon_i^2}{\mathrm{e}}^{-2/\epsilon_i},
\end{equation}
for some constant $C_i>0$.
Using \eqref{estimateomega},  straightforward algebra shows that the coefficient $A^0_{\epsilon_i}$ defined in \eqref{A0} is exponentially close to 1: there exists a constant $\tilde{C}_i>0$ such that%
\begin{equation}\label{coef}
\left|A^0_{\epsilon_i}-1\right| \leq \frac{\tilde{C}_i}{\epsilon_i}{\mathrm{e}}^{-1/\epsilon_i}.\nn
\end{equation}
Similarly,
using (\ref{estimateomega}) we obtain the decay rate of function $K_{-1,\epsilon}^0(t) $ defined in \eqref{K0}, as $\epsilon_i\to 0+$,
\begin{equation}\label{decay}
\epsilon_i\left(\omega^{(i)}\right)^2-\frac{1}{4\epsilon_i} = -\frac{1}{\epsilon_i}{\mathrm{e}}^{-1/\epsilon_i} + r_i + 2 {\mathrm{e}}^{-1/\epsilon_i} + \epsilon_i r_i^2<0.\nn
\end{equation}
Using \eqref{K0} and the last equation, we obtain 
\begin{eqnarray} \label{repG0}
K_{-1,\epsilon_i}^0(c_i^{-2}\epsilon_i t) &=& A^0_{\epsilon_i}\exp\left(-c_i^{-2}\textup{e}^{-1/\epsilon_i}t+(r_i+2\textup{e}^{-1/\epsilon_i}+\epsilon_ir_i^2)c_i^{-2}\epsilon_it\right),\nn
\end{eqnarray}
With this representation, it follows directly that
\begin{align}\label{estimate2}
\tilde{f}_{\mathcal{S}^{(0)},\hbar}=&\int_0^{\infty} K_{-1,\epsilon_1}^0(c_1^{-2}\epsilon_1t)K_{-1,\epsilon_2}^0(c_2^{-2}\epsilon_2t)dt\\
 = &\frac{A_{\epsilon_1}^0\cdot A_{\epsilon_2}^0}{\sum_{i=1}^2[c_i^{-2}\textup{e}^{-1/\epsilon_i}-(r_i+2\textup{e}^{-1/\epsilon_i}+\epsilon_ir_i^2)c_i^{-2}\epsilon_i]} =\frac{1}{\sum_{i=1}^2c_i^{-2}\textup{e}^{-1/\epsilon_i}} + {\mathcal{O}}(1/\epsilon_{\max}).\nn
\end{align}
This result now follows from  \eqref{estimate1} and \eqref{estimate2}.

\subsection{Proof of Proposition \ref{prop:bound_delay}}
Below we use Theorems \ref{ts_theorem_minus} and \ref{ts_theorem_plus} to prove Proposition \ref{prop:bound_delay}.

Using Lemma \ref{integral lemma}, we have
\begin{align}\label{bound_delay}
&|f_{\mathcal{S}^{(j)},\hbar}(0,0)-c_j^{2}(h_j-1+\textup{e}^{-h_j})|\\
=&\bigg|\int_0^\infty K_{+1,\epsilon_j}(c_j^{-2}\epsilon_jt,0)\cdot\bigg(\prod_{i\neq 1}K_{-1,\epsilon_i}(c_i^{-2}\epsilon_it,0)\bigg)\,dt-\int_0^\infty K_{+1,\epsilon_j}(c_j^{-2}\epsilon_jt,0)dt\bigg|\nonumber\\
\le&\int_0^{6c_{j}^2\epsilon_{j}^{-1}}K_{+1,\epsilon_j}(c_j^{-2}\epsilon_jt,0)\bigg|\prod_{i\neq j}K_{-1,\epsilon_i}(c_i^{-2}\epsilon_it,0)-1\bigg|\,dt+2\int_{6c_{j}^2\epsilon_{j}^{-1}}^\infty K_{+1,\epsilon_j}(c_j^{-2}\epsilon_jt,0)dt\nonumber\\
\le&\int_0^{6c_{j}^2\epsilon_{j}^{-1}}K_{+1,\epsilon_j}(c_j^{-2}\epsilon_jt,0)\bigg(1-\prod_{i\neq j}K_{-1,\epsilon_i}(c_i^{-2}\epsilon_it,0)\bigg)\,dt+\mathcal{O}(\frac{1}{\epsilon_{j}^{2}}\mathrm{e}^{-1/\epsilon_{j}})\nonumber\\
\le&\frac{6c_j^2}{\epsilon_j}\bigg(1-\prod_{i\neq j}K_{-1,\epsilon_i}\bigg(\frac{6c_i^{-2}\epsilon_i}{c_j^{-2}\epsilon_j},0\bigg)\bigg)+\mathcal{O}\bigg(\frac{1}{\epsilon_{j}^{2}}\mathrm{e}^{-1/\epsilon_{j}}\bigg),\nn
\end{align}
where we used Corollary \ref{Kbound} in the second inequality. The third inequality follows from the fact that $K_{-1,\epsilon_i}(t,0)$ is bounded between 0 and 1, and  it is decreasing in $t$. To bound the first term in \eqref{bound_delay}, we use Theorem \ref{ts_theorem_minus}, Corollary \ref{cor:Hbound}, \eqref{repG0} and the inequality $1-e^{-|x|}\le |x|$ to obtain that 
\begin{eqnarray*}
\bigg|1-\prod_{i\neq j}K_{-1,\epsilon_i}\bigg(\frac{6c_i^{-2}\epsilon_i}{c_j^{-2}\epsilon_j},0\bigg)\bigg|\le \sum_{i\neq j}\frac{6c_i^{-2}\mathrm{e}^{-1/\epsilon_i}}{c_j^{-2}\epsilon_j}(1+o(1))=\mathcal{O}\bigg(\frac{1}{\epsilon_j}\sum_{i\neq j}\mathrm{e}^{-1/\epsilon_i}\bigg).
\end{eqnarray*}
The result now follows.

{\bf Acknowledgements} The authors are grateful to the anonymous referee for pointing out the inequalities in (\ref{UB}) and (\ref{LB}) and thus giving solid ground for considering the optimization problem in (\ref{eqnproblemKLG}).

\end{document}